\documentclass[conference]{IEEEtran}
\usepackage{cite}
\usepackage{hyperref}
\usepackage{amsmath,amssymb}
\newtheorem{definition}{Definition}[section]
\newtheorem{proposition}[definition]{Proposition}
\newtheorem{lemma}[definition]{Lemma}
\newtheorem{theorem}[definition]{Theorem}
\newtheorem{corollary}[definition]{Corollary}
\newtheorem{remark}[definition]{Remark}
\newtheorem{example}[definition]{Example}
\newenvironment{proof}{\begin{IEEEproof}}{\end{IEEEproof}}
\newcommand{\Sect}{Section }
\usepackage[all]{xy}
\usepackage{tikz}
\usetikzlibrary{cd}
\usepackage{bussproofs}
\newcommand{\cat}[1]{\mathbb{#1}}
\newcommand{\arcat}[1]{#1^{\to}}
\newcommand{\fibcat}[1]{(#1)_{\text{f}}}
\newcommand{\Rel}{\mathbf{Rel}}
\newcommand{\cod}{\mathbf{cod}}
\newcommand{\Sph}[1]{\mathbb{S}^{#1}}
\newcommand{\Z}{\mathbb{Z}}
\newcommand{\N}{\mathbb{N}}
\newcommand{\Zero}{\mathbf{0}}
\newcommand{\One}{\mathbf{1}}
\newcommand{\Two}{\mathbf{2}}
\newcommand{\synt}[1]{\mathsf{#1}}
\newcommand{\kind}{\synt{kind}}
\newcommand{\type}{\synt{type}}
\newcommand{\refl}{\synt{refl}}
\newcommand{\Ind}{\synt{Ind}}
\newcommand{\base}{\synt{base}}
\newcommand{\Sloop}{\synt{loop}}
\newcommand{\U}{\mathcal{U}}
\newcommand{\pair}[2]{\langle #1, #2 \rangle}
\newcommand{\lang}[1]{\mathcal{L}(#1)}
\date{\today}
\title{Fibred Fibration Categories}
\hypersetup{
 pdfauthor={Taichi Uemura},
 pdftitle={Fibred Fibration Categories},
 pdfkeywords={},
 pdfsubject={},
 pdfcreator={Emacs 25.1.1 (Org mode 9.0.5)}, 
 pdflang={English}}
\begin{document}

\author{\IEEEauthorblockN{Taichi Uemura}
\IEEEauthorblockA{Research Institute for Mathematical Sciences,
Kyoto University\\
Email: \url{uemura@kurims.kyoto-u.ac.jp}\\
Website: \url{http://www.kurims.kyoto-u.ac.jp/~uemura}}}
\maketitle

\begin{abstract}
\label{sec-1}
We introduce fibred type-theoretic fibration categories
which are fibred categories between
categorical models of Martin-L\"{o}f type theory.
Fibred type-theoretic fibration categories give a categorical description of
logical predicates for identity types.
As an application,
we show a relational parametricity result for homotopy type theory.
As a corollary, it follows that
every closed term of type of polymorphic endofunctions on a loop space
is homotopic to some iterated concatenation of a loop.
\end{abstract}
\section{Introduction}
\label{sec:org650e83f}

\label{sec-2-1}

\label{sec-2-1-1}
\emph{Homotopy type theory} \cite{hottbook} is a recent subject
that combines Martin-L\"{o}f type theory \cite{martin-lof1975intuitionistic}
with homotopy theory.
The key idea is to identify types as spaces,
elements as points and equalities between elements as paths between points.
It provides an abstract language for
proving homotopy-theoretic theorems.
Among abstract languages of this sort,
including model categories
\cite{hovey2007model,quillen1967homotopical}
and various other models for (\(\infty\), 1)-categories
\cite{bergner2010survey,lurie2009higher},
homotopy type theory has unique tools
which are convenient to formalize homotopy theory.

\label{sec-2-1-2}
One such tool is \textbf{higher inductive types}.
They give a simple way to construct spaces
such as spheres, tori and other cell complexes.
We can define functions on a higher inductive type
``by recursion'' and
prove theorems on a higher inductive type
``by induction,''
just like functions and theorems on natural numbers.
Another tool is Voevodsky's \textbf{univalence axiom}.
Informally, this states that equivalent types are identical.
With this axiom we can prove
that isomorphic groups are equal,
equivalent categories are equal,
isometric Hilbert spaces are equal
and any other isomorphisms between mathematical structures
can be replaced by equalities.
Higher inductive types and the univalence axiom
provide a synthetic way to prove homotopy-theoretic theorems,
and the proofs are formalized in proof assistants
such as Coq and Agda \cite{bauer2017library,unimath,HoTT/HoTT-Agda}.

\label{sec-2-1-3}
Homotopy type theory provides such new technical tools
for algebraic topologists,
but what can we say about type theory itself?
In particular, how do the univalence axiom and higher inductive types
affect the behavior of type theory?
In the study of type theory,
``logical predicates'' have been a useful technique
for analyzing type theories.
We expect that this technique is useful for homotopy type theory,
but what is a logical predicate for homotopy type theory?
Shulman introduced \emph{type-theoretic fibration categories}
as sound and complete categorical semantics of Martin-L\"{o}f type theory
and proved that the \emph{gluing construction} \(\fibcat{\cat{D} \downarrow \Gamma}\)
for a suitable functor \(\Gamma : \cat{C} \to \cat{D}\)
between type-theoretic fibration categories
is again a type-theoretic fibration category \cite{shulman2015inverse}.
In his formulation,
a logical predicate is a section of a gluing construction
over the syntactic category.

\label{sec-2-2}
In this paper,
we give an alternative look at Shulman's formulation.
We regard \(\fibcat{\cat{D} \downarrow \Gamma}\) as a \emph{fibred category}
\(\fibcat{\cat{D} \downarrow \Gamma} \to \cat{C}\)
such that all fibers have type-theoretic structures
and the total type-theoretic structure is obtained from the fiberwise ones.
We show that
for a fibred category whose base category is a type-theoretic fibration category,
total and fiberwise structures of type-theoretic fibration category coincide under some conditions.
This gives a correct notion of fibred category between type-theoretic fibration categories.
We call such a fibred category a \emph{fibred type-theoretic fibration category}.

\label{sec-2-2-1}
Fibred categories are used as models of logical predicates
in the study of categorical type theory \cite{hermida1993fibrations,jacobs1999categorical}.
For a fibred category and
an interpretation of a type theory in the base category,
a logical predicate on the interpretation is
a cross-section of the fibred category over the syntactic category.
In our formulation,
a logical predicate for Martin-L\"{o}f type theory
is a cross-section of a fibred type-theoretic fibration category
over the syntactic category.
We have a similar formulation of logical predicates for homotopy type theory
by introducing the notion of \emph{fibred univalent universe}.

\label{sec-2-3}
As an application,
we show a relational parametricity result for homotopy type theory.
Relational parametricity is well-developed in the study of polymorphic type theory
\cite{bainbridge1990functorial,birkedal2005categorical,hasegawa1994categorical,ma1992types,plotkin1993logic,reynolds1983types,wadler1989free,wadler2003girard,wadler2007girard}.
Recently relational parametricity for dependent type theory
has been studied by several authors.
Krishnaswami and Dreyer \cite{dreyer2013internalizing}
and Atkey et al. \cite{atkey2014relationally} construct relationally parametric models
of the Calculus of Constructions and
Martin-L\"{o}f type theory respectively.
Takeuti \cite{takeuti2001parametricity} and Bernardy et al. \cite{bernardy2012proofs}
study relational parametricity
for the lambda cube and pure type systems respectively
via syntactic transformations from one type theory into another.
Following Takeuti and Bernardy et al.
we show the \emph{abstraction theorem} for homotopy type theory,
the soundness of a syntactic transformation of types to binary type families.
Our contribution is to give transformations of identity types
and the univalence axiom.
In the proof of the abstraction theorem,
we use a fibred type-theoretic fibration category
\(\Rel(\cat{T}) \to \cat{T}\) which we call the \emph{relational model}
for the syntactic category \(\cat{T}\).

As a corollary of the abstraction theorem,
we show the \emph{homotopy unicity property} on polymorphic functions
in homotopy type theory.
A typical example of polymorphic function in homotopy type theory is a function \(f\)
of the type \(\Pi_{X : \U}\Pi_{x : X}x = x \to x = x\).
This type cannot be defined in polymorphic type theory
because it uses dependent types,
and it seems to be trivial without homotopy-theoretic interpretation.
It follows from the abstraction theorem
that if \(f\) is a closed term of this type
then it must be homotopic to an iterated concatenation of a loop,
that is, for some integer \(n\), \(f(l) = l^{n}\) for all \(l : x = x\).
Note that, assuming the law of excluded middle,
this property does not hold
because the law of excluded middle allows case analysis on types.
We have a partial answer to the question how the univalence axiom affects the type theory:
the univalence axiom does not violate relational parametricity,
while the law of excluded middle does.
Further applications of the abstraction theorem
can be found in \cite{uemura2017homotopies}.

\label{sec-2-4}
There is some related work in the study of abstract homotopy theory.
Roig \cite{roig1994model} and Stanculescu \cite{stanculescu2012bifibrations}
considered weak factorization systems and model structures
on bifibred categories and
gave a construction from
fiberwise structures to total structures.
Szumi\l{}o \cite{szumilo2014two} introduced fibrations of (co)fibration categories
to study the homotopy theory of homotopy theories.

\label{sec-2-5}
\textbf{Organization.}
We begin in \Sect \ref{orgc341df6}
by recalling the definition and basic properties of
type-theoretic fibration categories.
In \Sect \ref{org3fd3ce8}
we define fibred type-theoretic fibration categories
and give two constructions of them.
One construction is the \emph{fiberwise-to-total construction} and
the other is the \emph{change of base}.
We also define the \emph{internal language}
for a fibred type-theoretic fibration category
and show the ``basic lemma'' for logical predicates.
In \Sect \ref{orgd199e01}
we discuss universes and the univalence axiom
in a fibred type-theoretic fibration category.
We construct a univalent universe in the total category
from fiberwise ones.
We also show that the univalence axiom
is preserved by change of base of fibred type-theoretic fibration categories.
Finally in \Sect \ref{orge32f29e}, we show a relational parametricity result for homotopy type theory
and its corollaries.

\section{Type-Theoretic Fibration Categories}
\label{sec:orgeb5b62f}
\label{orgc341df6}
First of all, we fix a notion of categorical models
of Martin-L\"{o}f's dependent type theory with
dependent product types, dependent sum types and identity types.
Among various categorical models of Martin-L\"{o}f type theory,
we use type-theoretic fibration category
because it seems to have the simplest formulation
of identity types.

\begin{definition}
\label{sec-3-1}
\label{orgcd183b9}\label{org6b06600}\label{orgb284de9}
Let \(i : A \to B\) and \(p : C \to D\) in a category.
The morphism \(i\) has the \emph{left lifting property with respect to \(p\)}
(\(p\) has the \emph{right lifting property with respect to \(i\)})
if for all \(f : A \to C\) and \(g : B \to D\)
such that \(p \circ f = g \circ i\),
there exists an \(h : B \to C\) such that
\(h \circ i = f\) and \(p \circ h = g\).
\end{definition}

\begin{definition}
\label{sec-3-2}
\label{org1cf19e2}
\cite[Definition 2.2]{shulman2016eidiagrams}
A \emph{type-theoretic fibration category} is
a category \(\cat{C}\) equipped with
a terminal object \(1\) and
a subcategory \(\mathcal{F} \subset \cat{C}\)
satisfying the conditions below.
Here a morphism in \(\mathcal{F}\) is called
a \emph{fibration} and denoted by
a two headed arrow \(A \twoheadrightarrow B\),
and a morphism that has the left lifting property
with respect to all fibrations
is called an \emph{acyclic cofibration} and denoted by
\(A \overset{\sim}{\rightarrowtail} B\).
\begin{enumerate}
\item \label{org57e04f0}
All isomorphisms and
all morphisms with codomain \(1\)
are fibrations.
\item \label{orgffa6b99}
Fibrations are closed under pullbacks:
if \(f : A \twoheadrightarrow B\) is a fibration
and \(s : B' \to B\) is any morphism,
then there exists a pullback \(s^{*}A\)
of \(A\) along \(s\),
and the morphism \(s^{*}A \to B'\) is again a fibration.
\item \label{orgd3d37e2}
Every morphism factors as
an acyclic cofibration
followed by a fibration.
\item \label{orge8efea6}For any fibrations \(f : A \twoheadrightarrow B\)
and \(g : B \twoheadrightarrow C\),
there exist a fibration
\(\Pi_{g}f : \Pi_{g}A \twoheadrightarrow C\)
and a natural bijection
\[ \cat{C}/C\left( h, \Pi_{g}f \right) \cong
   \cat{C}/B\left( g^{*}h, f \right) \]
for \(h : X \to C\).
Such a fibration \(\Pi_{g}f\) is called a \emph{dependent product of \(f\) along \(g\)}.
\end{enumerate}
\end{definition}

\begin{remark}
\label{sec-3-3}
The condition \ref{orge8efea6} of Definition \ref{org1cf19e2}
implies that the pullback along a fibration
preserves acyclic cofibrations.
\end{remark}

\begin{example}
\label{sec-3-4}
If a category \(\cat{C}\) has finite limits
and is locally cartesian closed,
then \(\cat{C}\) is a type-theoretic fibration category
where every morphism is a fibration.
In this case,
the acyclic cofibrations are the isomorphisms.
\end{example}
\begin{example}
\label{sec-3-5}
\label{org1eeac1b}
Let \(\cat{C}\) be a type-theoretic fibration category.
Write \(\fibcat{\cat{C} / A}\)
for the full subcategory of \(\cat{C} / A\)
where the objects are the fibrations over \(A\).
The category \(\fibcat{\cat{C} / A}\)
is a type-theoretic fibration category
whose fibrations are the morphisms
that are fibrations in \(\cat{C}\).
\end{example}
\begin{definition}
\label{sec-3-6}
A functor between type-theoretic fibration categories
is a \emph{type-theoretic functor} if
it preserves terminal objects, fibrations,
pullbacks of fibrations, acyclic cofibrations,
and dependent products.
\end{definition}
\begin{example}
\label{sec-3-7}
\label{orgaabe3e5}
For a type-theoretic fibration category \(\cat{C}\)
and its morphism \(s : A \to B\),
consider the pullback functor
\(s^{*} : \fibcat{\cat{C}/B}
\to \fibcat{\cat{C}/A}\).
It preserves fibrations by definition of type-theoretic fibration category
and acyclic cofibrations by the following Lemma \ref{org67ad797}.
Moreover it preserves dependent products by Lemma \ref{orgf15694c} below.
Thus \(s^{*}\) is a type-theoretic functor.
\end{example}

\begin{lemma}
\label{sec-3-8}
\label{org67ad797}
\cite[Lemma 2.3]{shulman2016eidiagrams}
Given the following diagram in a type-theoretic fibration category
\[
\begin{tikzcd}
 [row sep=3ex] s^{*}A \arrow[dd,twoheadrightarrow] \arrow[dr,"s^{*}f"'] \arrow[rr] &  & A \arrow[dd,twoheadrightarrow,"\hole" description] \arrow[dr,rightarrowtail,"f"] &  \\
  & s^{*}B \arrow[dl,twoheadrightarrow] \arrow[rr] &  & B \arrow[dl,twoheadrightarrow] \\
 C' \arrow[rr,"s"'] &  & C &  \end{tikzcd}
\]
where \(A \twoheadrightarrow C\) and \(B \twoheadrightarrow C\) are fibrations
and \(f : A \overset{\sim}{\rightarrowtail} B\) is an acyclic cofibration,
then \(s^{*}f : s^{*}A \to s^{*}B\) is an acyclic cofibration.
\end{lemma}

\begin{lemma}
\label{sec-3-9}
\label{orgf15694c}
The dependent products in a type-theoretic fibration category \(\cat{C}\)
satisfy the Beck-Chevalley condition:
if the following diagram is a pullback of a fibration
\[
\begin{tikzcd}
 B' \arrow[d,twoheadrightarrow,"g'"'] \arrow[r,"h"] & B \arrow[d,twoheadrightarrow,"g"] \\
 C' \arrow[r,"k"] & C \end{tikzcd}
\]
then the canonical natural transformation
\(k^{*}\Pi_{g} \Rightarrow \Pi_{g'}h^{*}\)
is an isomorphism.
\end{lemma}

\begin{proof}
\label{sec-3-10}
For any fibration \(f : A \twoheadrightarrow B\),
\(k^{*}\Pi_{g}A\) has the same universal property as \(\Pi_{g'}h^{*}A\).
Indeed, for any morphism \(x : X \to C'\) we have natural bijections
\begin{prooftree}
\AxiomC{\(x \to k^{*}\Pi_{g}f\) in \(\cat{C}/C'\)}
\UnaryInfC{\(k \circ x \to \Pi_{g}f\) in \(\cat{C}/C\)}
\UnaryInfC{\(g^{*}(k \circ x) \to f\) in \(\cat{C}/B\)}
\UnaryInfC{\(h \circ g'^{*}x \to f\) in \(\cat{C}/B\)}
\UnaryInfC{\(g'^{*}x \to h^{*}f\) in \(\cat{C}/B'\).}
\end{prooftree}
\end{proof}

\label{sec-3-11}
Type-theoretic fibration categories are
a categorical model of Martin-L\"{o}f type theory.
A type \(\Gamma \vdash A \ \type\)
is interpreted by a fibration \(A \twoheadrightarrow \Gamma\),
and a term \(\Gamma \vdash a : A\)
is interpreted by a section of the fibration \(A \twoheadrightarrow \Gamma\).
A dependent sum type \(\Gamma \vdash \Sigma_{a : A}B(a) \ \type\)
is interpreted by the composition of fibrations
\(B \twoheadrightarrow A \twoheadrightarrow \Gamma\).
A dependent product type \(\Gamma \vdash \Pi_{a : A}B(a) \ \type\)
is interpreted by a dependent product of \(B \twoheadrightarrow A\)
along \(A \twoheadrightarrow \Gamma\).
An identity type \(\Gamma, a : A, a' : A \vdash a = a' \ \type\)
is interpreted by a factorization
\(A \overset{\sim}{\rightarrowtail} P_{\Gamma}A \twoheadrightarrow A \times_{\Gamma} A\)
of the diagonal morphism \(A \to A \times_{\Gamma} A\).
Such an object \(P_{\Gamma}A\) is called a \emph{path object of \(A\) over \(\Gamma\)}.

\label{sec-3-12}
Let \(f, g : A \to B\) be parallel morphisms
between a morphism \(A \to \Gamma\) and a fibration \(B \twoheadrightarrow \Gamma\).
A \emph{homotopy from \(f\) to \(g\) over \(\Gamma\)}
is a morphism \(H : A \to P_{\Gamma}B\) into some path object
whose first and second projections are \(f\) and \(g\) respectively.
We say \(f\) and \(g\) are \emph{homotopic over \(\Gamma\)}, written \(f \sim_{\Gamma} g\),
when there exists a homotopy from \(f\) to \(g\) over \(\Gamma\).
We omit the subscript \(_{\Gamma}\) when \(\Gamma = 1\)
and write simply \(f \sim g\).
It is known that the relation \(\sim\) is a congruence relation on hom sets \cite[Section 3]{shulman2015inverse}.
A morphism \(f : A \to B\) is a \emph{homotopy equivalence}
if there exists a morphism \(g : B \to A\)
such that \(gf \sim 1\) and \(fg \sim 1\).
The morphism \(f\) is \emph{bi-invertible} if
there exist morphisms \(g, h : B \to A\)
such that \(gf \sim 1\) and \(fh \sim 1\).
The morphism \(f\) is a \emph{half adjoint equivalence}
if there exist a morphism \(g : B \to A\)
and homotopies \(\eta : gf \sim 1\) and \(\varepsilon : fg \sim 1\)
such that \(f\eta \sim_{B \times B} \varepsilon f\).
By the standard argument in homotopy type theory
the notions of homotopy equivalences, bi-invertible morphisms
and half adjoint equivalences are logically equivalent \cite[Chapter 4]{hottbook}.

\label{sec-3-13}
We state some properties of type-theoretic fibration categories
for future use.
\begin{lemma}
\label{sec-3-13-1}
\label{org8c01c79}
If \(f : A \overset{\sim}{\rightarrowtail} B\)
is an acyclic cofibration,
then there exists a morphism \(g : B \to A\)
such that \(gf = 1\).
\end{lemma}

\begin{proof}
\label{sec-3-13-2}
Use the lifting property with respect to
the fibration \(A \twoheadrightarrow 1\).
\end{proof}

\begin{lemma}
\label{sec-3-13-3}
\label{org5a8693c}
\cite[Lemma 2.4]{shulman2016eidiagrams}
For \(f : A \to B\) and \(g : B \to C\),
if \(g\) and \(gf\) are acyclic cofibrations,
then so is \(f\).
\end{lemma}

\begin{lemma}
\label{sec-3-13-4}
\label{orgb900b60}
\cite[Lemma 12.2]{shulman2015inverse}
Let \(F : \cat{C} \to \cat{D}\) be a functor
between type-theoretic fibration categories.
If \(F\) preserves fibrations, pullbacks of fibrations and acyclic cofibrations,
then it preserves homotopy equivalences.
\end{lemma}

\section{Fibred Type-Theoretic Fibration Categories}
\label{sec:org9d492c8}
\label{org3fd3ce8}
In this section we introduce
fibred type-theoretic fibration categories
and give standard constructions of them.
One construction is the \emph{fiberwise-to-total construction}:
given a fibred category whose base category and
all fibers are type-theoretic fibration categories,
we can make the total category a type-theoretic fibration category
under some extra conditions.
Another construction is the \emph{change of base}.
Change of base yields several examples of fibred type-theoretic fibration category,
as we shall see in Example \ref{orgf4dcb7d} and \ref{org38edd10}.
We also define the \emph{internal language} for a fibred type-theoretic fibration category
and explain the syntactic intuition of fibred type-theoretic fibration categories.
Finally we give a categorical definition of \emph{logical predicates}
for Martin-L\"{o}f type theory
and prove the ``basic lemma'' for logical predicates.

\begin{definition}
\label{sec-4-1}
\label{org204acfc}\label{orgec620d6}
A \emph{fibred type-theoretic fibration category} is a
fibred category \cite[Definition 1.1.3]{jacobs1999categorical}
\(p : \cat{E} \to \cat{B}\)
satisfying the following conditions.
\begin{enumerate}
\item The categories \(\cat{E}\) and \(\cat{B}\) are type-theoretic fibration categories
and \(p\) is a type-theoretic functor.
\item \label{org6d6527a}Every cartesian morphism above
a fibration
is a fibration.
\item \label{org2f1c729}In the following diagram
in \(\cat{E} \to \cat{B}\)
\[
\begin{tikzcd}
 [row sep=2ex] A \arrow[drr,twoheadrightarrow,"g"] \arrow[dr,dashrightarrow] &  &  \\
  & t^{*}C \arrow[r] & C \\
 I \arrow[dr,twoheadrightarrow,"s"'] \arrow[drr,twoheadrightarrow] &  &  \\
  & J \arrow[r,"t"'] & K, \end{tikzcd}
\]
if \(g\) and \(s\) are fibrations,
then the induced morphism
\(A \to t^{*}C\) above \(s\)
is a fibration.
\item \label{orgee4a1a2}Every cartesian morphism above
an acyclic cofibration
is an acyclic cofibration.
\end{enumerate}
\end{definition}

\subsection{Fiberwise-to-Total Construction}
\label{sec:orgd50ce84}

\begin{definition}
\label{sec-4-2-1}
Let \(p : \cat{C} \to \cat{D}\) be a functor.
A morphism \(f : A \to B\) in \(\cat{C}\) is \emph{weakly \(p\)-cartesian}
if for any morphisms \(g : X \to B\) and \(s : pX \to pA\)
such that \(pf \circ s = pg\),
there exists a morphism \(h : X \to A\)
such that \(f \circ h = g\) and \(ph = s\).
\end{definition}
\begin{theorem}
\label{sec-4-2-2}
\label{org24a9bbb}\label{org4e7cd71}
Suppose \(p : \cat{E} \to \cat{B}\)
is a fibred category satisfying the following conditions.
\begin{enumerate}
\item The base category \(\cat{B}\) and
all fibers \(\cat{E}_{I}\)
are type-theoretic fibration categories.
\item \label{org7c2e30a}
For every morphism \(s : I \to J\) in \(\cat{B}\),
the reindexing functor
\(s^{*} : \cat{E}_{J} \to \cat{E}_{I}\)
is a type-theoretic functor.
\item \label{orgd4716d9}\label{orga40836e}
For every acyclic cofibration
\(s : I \overset{\sim}{\rightarrowtail} J\) in \(\cat{B}\),
every fibration in \(\cat{E}_{J}\)
is weakly \(s^{*}\)-cartesian.
\item \label{org45c198d}
For any fibration \(s : I \twoheadrightarrow J\) in \(\cat{B}\),
the reindexing functor \(s^{*} : \cat{E}_{J} \to \cat{E}_{I}\)
has a right adjoint \(s_{*}\) preserving fibrations.
Moreover, the Beck-Chevalley condition holds:
if the following diagram is a pullback of a fibration
\[
\begin{tikzcd}
 I' \arrow[d,twoheadrightarrow,"s'"'] \arrow[r,"t"] & I \arrow[d,twoheadrightarrow,"s"] \\
 J' \arrow[r,"r"'] & J, \end{tikzcd}
\]
then the canonical natural transformation
\(r^{*}s_{*} \Rightarrow s'_{*}t^{*}\)
is an isomorphism.
\end{enumerate}
Then \(\cat{E}\) has a structure of
type-theoretic fibration category whose
fibrations are the Reedy fibrations defined below,
and \(p\) is a fibred type-theoretic fibration category.
\end{theorem}

\begin{example}
\label{sec-4-2-3}
Let \(\cat{C}\) be a type-theoretic fibration category.
Write \(\fibcat{\arcat{\cat{C}}}\) for
the full subcategory of the arrow category \(\arcat{\cat{C}}\)
where the objects are the fibrations in \(\cat{C}\).
Consider the codomain functor
\(\cod : \fibcat{\arcat{\cat{C}}} \to \cat{C}\).
This functor satisfies the hypotheses of Theorem \ref{org4e7cd71}.
Its fiber at an object \(A\) is \(\fibcat{\cat{C}/A}\)
and is a type-theoretic fibration category
as in Example \ref{org1eeac1b}.
By Example \ref{orgaabe3e5},
a reindexing functor is a type-theoretic functor.
The condition \ref{orgd4716d9} follows from
the fact that a pullback of an acyclic cofibration along
a fibration is an acyclic cofibration.
Finally, a right adjoint to the pullback functor along a fibration
is given by the dependent product.
Thus the codomain functor is a fibred type-theoretic fibration category.
A morphism \((f_{1}, f_{0}) : (\alpha : A_{1} \twoheadrightarrow A_{0})
\to (\beta : B_{1} \twoheadrightarrow B_{0})\)
in \(\fibcat{\arcat{\cat{C}}}\) is a fibration if and only if
\(f_{0} : A_{0} \to B_{0}\) is a fibration and
the induced morphism \(A_{1} \to B_{1} \times_{A_{0}} A_{1}\)
is a fibration.

Note that the fact that \(\fibcat{\arcat{\cat{C}}}\)
is a type-theoretic fibration category was
originally proved by Shulman \cite[Theorem 8.8]{shulman2015inverse}.
Our contribution is to give a fibred categorical description
for the construction.
\end{example}

\label{sec-4-2-4}
In the rest of the section
we prove Theorem \ref{org4e7cd71}.
Throughout the section
we assume \(p : \cat{E} \to \cat{B}\)
is a fibred category satisfying all the hypotheses
of Theorem \ref{org4e7cd71}.

\begin{definition}
\label{sec-4-2-4-1}
\label{orga37332a}\label{org16ed92e}
Let \(f : A \to B\) be a morphism in \(\cat{E}\).
The morphism \(f\) is a \emph{Reedy fibration} if
\(pf\) is a fibration in \(\cat{B}\)
and the induced morphism
\(A \to \left( pf \right)^{*}B\)
is a fibration in \(\cat{E}_{pA}\).
The morphism \(f\) is a \emph{Reedy acyclic cofibration}
if \(pf\) is
an acyclic cofibration in \(\cat{B}\) and
\(f\) factors as a cartesian morphism above \(pf\)
followed by an acyclic cofibration in \(\cat{E}_{pB}\).
\[
\begin{tikzcd}
 [row sep=3ex] A \arrow[dr,"f"] \arrow[d,twoheadrightarrow] &  &  & B \\
 (pf)^{*}B \arrow[r] & B & A \arrow[ur,"f"] \arrow[r,"\text{Cart.}"'] & C \arrow[u,rightarrowtail,"\sim","g"'] \\
 pA \arrow[r,twoheadrightarrow,"pf"] & pB & pA \arrow[r,rightarrowtail,"\sim","pf"'] & pB \end{tikzcd}
\]
\end{definition}

\begin{remark}
\label{sec-4-2-4-2}
This definition of Reedy fibrations coincide with Shulman's \cite[Definition 8.1]{shulman2015inverse}
in the case that \(p\) is a codomain functor
\(\fibcat{\arcat{\cat{C}}} \to \cat{C}\).
\end{remark}

\begin{lemma}
\label{sec-4-2-4-3}
\label{org74ddcdc}
Suppose \(s : I \overset{\sim}{\rightarrowtail} J\)
is an acyclic cofibration in \(\cat{B}\) and
\(A\) is an object above \(I\).
Then there exist an object \(s_{!}A\) above \(J\) and
a cartesian morphism \(\eta : A \to s_{!}A\) above \(s\).
\end{lemma}
\begin{proof}
\label{sec-4-2-4-4}
By Lemma \ref{org8c01c79},
there exists a morphism
\(t : J \to I\) such that \(ts = 1\).
Therefore \(A \cong s^{*}t^{*}A\) above \(I\), and thus
we have a cartesian morphism
\(A \to t^{*}A\) above \(s\).
\end{proof}
\begin{lemma}
\label{sec-4-2-4-5}
\label{org5d31cb4}
Every cartesian morphism above an acyclic cofibration
has the left lifting property
with respect to all the vertical fibrations:
for every acyclic cofibration
\(s : I \overset{\sim}{\rightarrowtail} J\)
and cartesian morphism
\(f : A \to B\) above \(s\),
if \(g : C \twoheadrightarrow D\)
is a fibration in the following diagram
in \(\cat{E} \to \cat{B}\)
\[
\begin{tikzcd}
 [row sep=3ex] &  & C \arrow[d,twoheadrightarrow,"g"] \\
 A \arrow[r,"f"'] \arrow[urr] & B \arrow[r] & D \\
 I \arrow[r,rightarrowtail,"\sim","s"'] & J \arrow[r] & K, \end{tikzcd}
\]
then there exists a filling morphism
\(h : B \to C\).
\end{lemma}
\begin{proof}
\label{sec-4-2-4-6}
By reindexing, we can assume \(K = J\) and \(J \to K\) is the identity.
Then the statement is equivalent to the condition \ref{orgd4716d9}.
\end{proof}
\begin{lemma}
\label{sec-4-2-4-7}
\label{org1b5b6ac}
Every morphism \(f : A \to B\) in \(\cat{E}\)
factors as
a Reedy acyclic cofibration followed by
a  Reedy fibration.
\end{lemma}
\begin{proof}
\label{sec-4-2-4-8}
Let \(s = pf : I \to J\).
Then \(s = r \circ t\) for some acyclic cofibration
\(t : I \overset{\sim}{\rightarrowtail} K\) and fibration
\(r : K \twoheadrightarrow J\).
Using the lifting property in Lemma \ref{org5d31cb4}
with respect to the fibration \(r^{*}B \twoheadrightarrow 1\)
in \(\cat{E}_{K}\),
we have a morphism
\(t_{!}A \to r^{*}B\) in
\(\cat{E}_{K}\) such that the following diagram commutes
\[
\begin{tikzcd}
 [row sep=3ex] & t_{!}A \arrow[dd,dashrightarrow] &  \\
 A \arrow[ur] \arrow[dr] \arrow[rr,near end,"f"] &  & B \\
  & r^{*}B \arrow[ur] &  \\
 I \arrow[r,rightarrowtail,"\sim","t"'] & K \arrow[r,twoheadrightarrow,"r"'] & J. \end{tikzcd}
\]
Taking a factorization in \(\cat{E}_{K}\),
we have an acyclic cofibration
\(t_{!}A \overset{\sim}{\rightarrowtail} C\)
followed by a fibration
\(C \twoheadrightarrow r^{*}B\).
The morphisms \(A\) to \(C\)
and \(C\) to \(B\) give the desired factorization.
\end{proof}
\begin{lemma}
\label{sec-4-2-4-9}
\label{orgef4cf25}\label{org1e0b166}
The total category \(\cat{E}\) has a dependent product of a fibration
along a fibration.
\end{lemma}
\begin{proof}
\label{sec-4-2-4-10}
Suppose \(f : A \twoheadrightarrow B\) and \(g : B \twoheadrightarrow C\)
are fibrations above
\(s : I \twoheadrightarrow J\) and \(t : J \twoheadrightarrow K\) respectively.
We construct a fibration \(\Pi_{g}f : \Pi_{g}A \twoheadrightarrow C\)
above \(\Pi_{t}s : \Pi_{t}I \twoheadrightarrow K\).
First we get a dependent product
\(\Pi_{s^{*}\langle g \rangle}\langle f \rangle : \Pi_{s^{*}\langle g \rangle}A \twoheadrightarrow s^{*}t^{*}C\) above \(I\)
where \(\langle f \rangle : A \twoheadrightarrow s^{*}B\) and \(\langle g \rangle : B \twoheadrightarrow t^{*}C\)
are the morphisms induced by cartesianness.
Let \(\varepsilon : t^{*}\Pi_{t}I \to I\)
be the counit of the adjunction \(t^{*} \dashv \Pi_{t}\)
and \(\bar{t} : t^{*}\Pi_{t}I \twoheadrightarrow \Pi_{t}I\)
the upper fibration of the pullback of \(\Pi_{t}I\) along \(t\).
\[
\begin{tikzcd}
 [row sep=3ex] t^{*} \Pi_{t}I \arrow[r,twoheadrightarrow,"\bar{t}"] \arrow[d,"\varepsilon"'] & \Pi_{t}I \arrow[dd,twoheadrightarrow,"\Pi_{t}s"] \\
 I \arrow[d,twoheadrightarrow,"s"'] &  \\
 J \arrow[r,twoheadrightarrow,"t"'] & K \end{tikzcd}
\]
Since \(ts\varepsilon = (\Pi_{t}s)\bar{t}\),
we have \(\varepsilon^{*}s^{*}t^{*}C \cong \bar{t}^{*}(\Pi_{t}s)^{*}C\).
Thus we have the unit
\(\eta : (\Pi_{t}s)^{*}C \to \bar{t}_{*}\varepsilon^{*}s^{*}t^{*}C\)
of the adjunction \(\bar{t}^{*} \dashv \bar{t}_{*}\).
Then the composition
\(\eta^{*}\bar{t}_{*}\varepsilon^{*}\Pi_{s^{*}\langle g \rangle}A
\twoheadrightarrow \left( \Pi_{t}s \right)^{*}C
\twoheadrightarrow C\)
is a Reedy fibration.
We show that it satisfies the universal property of dependent product.

Let \(x : X \to C\) be a morphism in \(\cat{E}\) above \(a : L \to K\).
Let \(b : L \to \Pi_{t}I\) be a morphism over \(K\)
which corresponds to \(\hat{b} : t^{*}L \to I\) over \(J\)
as \(\varepsilon \circ t^{*}b = \hat{b}\).
The goal is to construct a natural bijection between
the set of morphisms \(X \to \eta^{*}\bar{t}_{*}\varepsilon^{*}\Pi_{s^{*}\langle g \rangle}A\)
above \(b\) over \(C\) and
the set of morphisms \(g^{*}X \to A\)
above \(\hat{b}\) over \(B\).
A morphism \(y : X \to \eta^{*}\bar{t}_{*}\varepsilon^{*}\Pi_{s^{*}\langle g \rangle}A\)
above \(b\) over \(C\) corresponds to a morphism \(y_{1} : X \to b^{*}\bar{t}_{*}\varepsilon^{*}\Pi_{s^{*}\langle g \rangle}A\)
above \(L\) such that
\[
\begin{tikzcd}
 X \arrow[r,"y_{1}"] \arrow[d] & b^{*}\bar{t}_{*} \varepsilon^{*} \Pi_{s^{*}\langle g \rangle}A \arrow[d,twoheadrightarrow] \\
 b^{*}(\Pi_{t}s)^{*}C \ar[r,"b^{*} \eta"'] & b^{*}\bar{t}_{*}\bar{t}^{*}(\Pi_{t}s)^{*}C \end{tikzcd}
\]
commutes.
Consider the pullbacks
\[
\begin{tikzcd}
 t^{*}L \ar[r]^-{\bar{b}} \ar[d,twoheadrightarrow,"\tilde{t}"'] & t^{*} \Pi_{t}I \arrow[r,twoheadrightarrow] \ar[d,twoheadrightarrow,"\bar{t}"'] & J \arrow[d,twoheadrightarrow,"t"] \\
 L \arrow[r,"b"'] & \Pi_{t}I \arrow[r,twoheadrightarrow,"\Pi_{t}s"'] & K. \end{tikzcd}
\]
Then \(b^{*}\eta\) is isomorphic to
\(\eta_{b^{*}(\Pi_{t}s)^{*}C} : b^{*}(\Pi_{t}s)C \to \tilde{t}_{*}\tilde{t}^{*}b^{*}(\Pi_{t}s)^{*}C\)
along \(b^{*}\bar{t}_{*}\bar{t}^{*} \cong \tilde{t}_{*}\bar{b}^{*}\bar{t}^{*} \cong \tilde{t}_{*}\tilde{t}^{*}b^{*}\)
by the Beck-Chevalley condition.
Thus the morphism \(y_{1}\) corresponds to
a morphism \(y_{2} : \tilde{t}^{*}X \to \bar{b}^{*}\varepsilon^{*}\Pi_{s^{*}\langle g \rangle}A\)
over \(\tilde{t}^{*}b^{*}(\Pi_{t}s)^{*}C
\cong \bar{b}^{*}\bar{t}^{*}(\Pi_{t}s)^{*}C
\cong \bar{b}^{*}\varepsilon^{*}s^{*}t^{*}C\).
Since \(\varepsilon \circ \bar{b} = \hat{b}\),
\(y_{2}\) corresponds to a morphism
\(y_{3} : \tilde{t}^{*}X \to \hat{b}^{*}\Pi_{s^{*}\langle g \rangle}A\)
over \(\hat{b}^{*}s^{*}t^{*}C\).
By assumption the reindexing functor \(\hat{b}^{*}\) preserves dependent products,
and thus \(\hat{b}^{*}\Pi_{s^{*}\langle g \rangle}A
\cong \Pi_{\hat{b}^{*}s^{*}\langle g \rangle}(\hat{b}^{*}A)\).
Since \(s \circ \hat{b} = t^{*}a : t^{*}L \to J\),
\(y_{3}\) corresponds to a morphism
\(y_{4} : (\bar{a}^{*}\langle g \rangle)^{*}\tilde{t}^{*}X \to \hat{b}^{*}A\)
over \(\hat{b}^{*}s^{*}B\) where \(\bar{a} = t^{*}a\).
Observe that \((\bar{a}^{*}\langle g \rangle)^{*}\tilde{t}^{*}X\)
is a pullback of \(X\) along \(g : B \twoheadrightarrow C\) in \(\cat{E}\).
Thus \(y_{4}\) corresponds to a morphism \(\hat{y} : g^{*}X \to A\)
above \(\hat{b}\) over \(B\).
\end{proof}
\begin{proof}[{Proof of Theorem \ref{org24a9bbb}}]
\label{sec-4-2-4-11}
Clearly \(\cat{E}\) has a terminal object and
Reedy fibrations are closed under pullbacks.
We can easily show that
Reedy acyclic cofibrations have the left lifting property
with respect to all the Reedy fibrations.
Thus Lemmas
\ref{org1b5b6ac} and \ref{orgef4cf25}
show that \(\cat{E}\)
is a type-theoretic fibration category whose
fibrations are the Reedy fibrations.
The other conditions of Definition \ref{org204acfc}
are clear by definition.
\end{proof}

\subsection{Basic Properties}
\label{sec:org0d18755}
We give some basic properties of fibred type-theoretic fibration categories.
Throughout the section,
let \(p : \cat{E} \to \cat{B}\) be
a fibred type-theoretic fibration category.
\begin{proposition}
\label{sec-4-3-1}
\label{org4fa27cf}
Given the following diagram in \(\cat{E} \to \cat{B}\)
\[
\begin{tikzcd}
 [row sep=2ex] A \arrow[rr] \arrow[dr,rightarrowtail,"\sim"] &  & C \arrow[dr,twoheadrightarrow] &  \\
  & B \arrow[rr] &  & D \\
 I \arrow[rr,"s"] \arrow[dr,rightarrowtail,"\sim"] &  & K \arrow[dr,twoheadrightarrow] &  \\
  & J \arrow[rr,"t"'] \arrow[ur,"r"'] &  & L, \end{tikzcd}
\]
then there exists a filling morphism \(h : B \to C\)
above \(r\).
\end{proposition}
\begin{proof}
\label{sec-4-3-2}
By reindexing, we have the following diagram
in \(\cat{E} \to \cat{B}\)
\[
\begin{tikzcd}
 [row sep=2ex] A \arrow[rr] \arrow[dr,rightarrowtail,"\sim"] &  & r^{*}C \arrow[dr,twoheadrightarrow] &  \\
  & B \arrow[rr] &  & t^{*}D \\
 I \arrow[rr,rightarrowtail,"\sim"] \arrow[dr,rightarrowtail,"\sim"'] &  & J \arrow[dr,equal] &  \\
  & J \arrow[rr,equal] &  & J. \end{tikzcd}
\]
Then we have a filling morphism \(k : B \to r^{*}C\),
and \(pk\) must be the identity.
The composition \(B \to r^{*}C \to C\)
is a filling morphism above \(r\).
\end{proof}
\begin{lemma}
\label{sec-4-3-3}
\label{org4b23fdf}
Let \(f : A \to B\) be a morphism in a fiber \(\cat{E}_{I}\).
Then \(f\) factors as an acyclic cofibration above \(I\)
followed by a fibration above \(I\).
\end{lemma}
\begin{proof}
\label{sec-4-3-4}
First we have the following factorization in \(\cat{E}\)
\[
\begin{tikzcd}
 [row sep=3ex] A \arrow[rr,"f"] \arrow[dr,rightarrowtail,"\sim"] &  & B \\
  & C \arrow[ur,twoheadrightarrow] &  \\
 I \arrow[rr,equal] \arrow[dr,rightarrowtail,"\sim","s"'] &  & I \\
  & J. \arrow[ur,twoheadrightarrow,"t"'] &  \end{tikzcd}
\]
By Lemma \ref{org5a8693c},
the induced morphism \(A \to s^{*}C\) is an acyclic cofibration.
Therefore we have the following factorization in \(\cat{E}_{I}\)
\[
\begin{tikzcd}
 A \arrow[r,rightarrowtail,"\sim"] \arrow[rrr,bend right=15,"f"'] & s^{*}C \arrow[r,twoheadrightarrow] & s^{*}t^{*}B \arrow[r,"\cong"] & B. \end{tikzcd}
\]
\end{proof}
\begin{proposition}
\label{sec-4-3-5}
\label{orgaf904b5}
Let \(f : A \to B\) be a morphism in \(\cat{E}\) and
\[
\begin{tikzcd}
 pA \arrow[r,rightarrowtail,"\sim","s"'] & K \arrow[r,twoheadrightarrow,"t"'] & pB \end{tikzcd}
\]
be a factorization of \(pf\).
Then there exists a factorization
\[
\begin{tikzcd}
 A \arrow[r,rightarrowtail,"\sim","g"'] & C \arrow[r,twoheadrightarrow,"h"] & B \end{tikzcd}
\]
above \((s, t)\).
\end{proposition}
\begin{proof}
\label{sec-4-3-6}
As in Lemma \ref{org74ddcdc},
we have a cartesian morphism \(\eta : A \to s_{!}A\) above \(s\).
By Proposition \ref{org4fa27cf},
applied for the fibration \(t^{*}B \twoheadrightarrow 1_{K}\),
we have a factorization of \(f\)
\[
\begin{tikzcd}
 A \arrow[r,rightarrowtail,"\sim","\eta"'] & s_{!}A \arrow[r,"k"] & t^{*}B \arrow[r,twoheadrightarrow,"\varepsilon"'] & B \end{tikzcd}
\]
where \(\eta\), \(k\) and \(\varepsilon\) are above
\(s\), \(K\) and \(t\) respectively.
Take a vertical factorization of \(k\)
using Lemma \ref{org4b23fdf}.
\end{proof}
\subsection{Change of Base}
\label{sec:org95dc807}
\begin{proposition}
\label{sec-4-4-1}
\label{orgc8a114a}
Let \(p : \cat{E} \to \cat{B}\)
be a fibred type-theoretic fibration category,
\(\cat{A}\) a type-theoretic fibration category and
\(F : \cat{A} \to \cat{B}\) a functor
preserving fibrations, pullbacks of fibrations and
acyclic cofibrations.
Then the change of base \(F^{*}\cat{E} \to \cat{A}\)
of \(p\) along \(F\) is
a fibred type-theoretic fibration category where
the fibrations are the levelwise fibrations:
a morphism \((s, f) : (I, A) \to (J, B)\) in \(F^{*}\cat{E}\)
is a fibration if and only if \(s\) and \(f\) are fibrations.
\end{proposition}
\begin{proof}
\label{sec-4-4-2}
We first show that the total category \(F^{*}\cat{E}\)
is a type-theoretic fibration category.
The object \((1_{\cat{A}}, !_{F(1_{\cat{A}})}^{*}1_{\cat{E}})\)
is a terminal object in \(F^{*}\cat{E}\)
where \(!_{F(1_{\cat{A}})} : F(1_{\cat{A}}) \to 1_{\cat{B}}\)
is the unique morphism to the terminal object in \(\cat{B}\).
The unique morphism \((I, X) \to (1_{\cat{A}}, !_{F(1_{\cat{A}})}^{*}1_{\cat{E}})\)
is a levelwise fibration by the condition \ref{org2f1c729} of Definition \ref{orgec620d6}.
A pullback of a fibration is calculated
by levelwise pullbacks.
Thus \(F^{*}\cat{E}\) satisfies
the conditions \ref{org57e04f0} and \ref{orgffa6b99} of Definition \ref{org1cf19e2}

Proposition \ref{orgaf904b5} and
preservation of acyclic cofibrations and fibrations
imply that every morphism in \(F^{*}\cat{E}\)
factors as a levelwise acyclic cofibration
followed by a levelwise fibration.
Proposition \ref{org4fa27cf} shows that
any levelwise acyclic cofibration has
the left lifting property with respect to
all the levelwise fibrations.
Therefore the condition \ref{orgd3d37e2} holds.

To construct a dependent product,
let \(s : I \twoheadrightarrow J\) and \(t : J \twoheadrightarrow K\)
be fibrations in \(\cat{A}\) and
\(f : A \twoheadrightarrow B\) and \(g : B \twoheadrightarrow C\)
be fibrations above \(Fs\) and \(Ft\) respectively.
Let \(\varepsilon_{0} : t^{*}\Pi_{t}I \to I\)
be the counit of the adjunction \(t^{*} \dashv \Pi_{t}\).
Then we have a morphism
\(\varepsilon' : F(\Pi_{t}I) \to \Pi_{Ft}FI\) over \(FK\)
which corresponds to
\((Ft)^{*}F(\Pi_{t}I) \cong F(t^{*}\Pi_{t}I) \overset{F\varepsilon_{0}}{\to} FI\).
The composition \(\varepsilon'^{*}\Pi_{g}A \to \Pi_{g}A \to C\)
is a fibration because it factors in \(\cat{E} \to \cat{B}\) as
\[
\begin{tikzcd}
 [row sep=3ex] \varepsilon'^{*} \Pi_{g}A \arrow[r,"\bar{\varepsilon'}"] \arrow[d,twoheadrightarrow] & \Pi_{g}A \arrow[dr,twoheadrightarrow] \arrow[d,twoheadrightarrow] &  \\
 (F(\Pi_{t}s))^{*}C \arrow[r] \arrow[rr,twoheadrightarrow,bend right=15] & (\Pi_{Ft}(Fs))^{*}C \arrow[r,twoheadrightarrow] & C \\
 F(\Pi_{t}I) \arrow[r,"\varepsilon'"] \arrow[rr,twoheadrightarrow,bend right=15,"F(\Pi_{t}{s})"'] & \Pi_{Ft}FI \arrow[r,twoheadrightarrow,"\Pi_{Ft}(Fs)"] & FK \end{tikzcd}
\]
where all arrows in the middle row and \(\bar{\varepsilon'}\)
are cartesian morphisms
so that the upper left square is a pullback.
It is easy to check that
the fibration \((\Pi_{t}I, \varepsilon'^{*}\Pi_{g}A) \twoheadrightarrow (K, C)\)
is a dependent product of \((s, f)\) along \((t, g)\).
Hence the total category \(F^{*}\cat{E}\)
is a type-theoretic fibration category.

To show that \(F^{*}\cat{E} \to \cat{A}\)
is a fibred type-theoretic fibration category
is easy and left to the reader.
\end{proof}
\begin{example}
\label{sec-4-4-3}
\label{orgf4dcb7d}
Let \(F : \cat{A} \to \cat{B}\)
be a functor between type-theoretic fibration categories
preserving fibrations, pullbacks of fibrations and acyclic cofibrations.
The change of base of \(\cod : \fibcat{\arcat{\cat{B}}} \to \cat{B}\)
along \(F\) is called the \emph{gluing construction for \(F\)}.
Its total category, written as \(\fibcat{\cat{B} \downarrow F}\),
is same as Shulman's construction \cite[Section 13]{shulman2015inverse}.
It is the full subcategory of the comma category \((\cat{B} \downarrow F)\)
where the objects are the triples of
objects \(B \in \cat{B}\) and \(A \in \cat{A}\) and
a fibration \(u : B \twoheadrightarrow FA\).
\end{example}
\begin{example}
\label{sec-4-4-4}
\label{org38edd10}
Let \(p : \cat{E} \to \cat{B}\)
be a fibred type-theoretic fibration category.
The change of base of \(p\) along
the functor \(I \mapsto I \times I\)
is called the \emph{relational model for \(p\)},
and we write the total category as \(\Rel(p)\).
Its objects are the pairs of
objects \(I \in \cat{B}\) and \(A \in \cat{E}_{I \times I}\).
We can think of \(\Rel(p)\)
as the category of binary relations or binary type families.
\end{example}

\label{sec-4-4-5}
As a corollary of change of base,
we have the converse of Theorem \ref{org4e7cd71}.
Since structure of type-theoretic fibration category is determined by its fibrations,
the constructions of Theorem \ref{org4e7cd71} and
Theorem \ref{org5f7bf15} below are mutually inverse.

\begin{theorem}
\label{sec-4-4-5-1}
\label{org5f7bf15}
Let \(p : \cat{E} \to \cat{B}\) be
a fibred type-theoretic fibration category.
Then each fiber \(\cat{E}_{I}\) is a type-theoretic fibration category,
restricting the structure of \(\cat{E}\),
and all the hypotheses of Theorem \ref{org4e7cd71} hold.
\end{theorem}

\begin{proof}
\label{sec-4-4-5-2}
By Proposition \ref{orgc8a114a},
each fiber \(\cat{E}_{I}\) is a type-theoretic fibration category.
We check the other conditions of Theorem \ref{org4e7cd71}.

Observe that the morphism part of the reindexing functor along a morphism \(s\)
is given by the pullback squares
\[
\begin{tikzcd}
 s^{*}A \arrow[r] \arrow[d,"s^{*}f"'] & A \arrow[d,"f"] \\
 s^{*}B \arrow[r] & B \end{tikzcd}
\]
in \(\cat{E}\).
Thus the reindexing functor preserves
fibrations and pullbacks of fibrations.
Using Lemma \ref{org67ad797}
it also preserves acyclic cofibrations.
It also preserves terminal objects
by the construction of terminal objects in the proof of Proposition \ref{orgc8a114a}.
Using Lemma \ref{orgf15694c}
it preserves dependent products.
Thus the condition \ref{org7c2e30a} holds.
Since a cartesian morphism above an acyclic cofibration
is an acyclic cofibration,
the condition \ref{orgd4716d9} holds.

To show the condition \ref{org45c198d},
let \(s : I \twoheadrightarrow J\) be a fibration in \(\cat{B}\)
and \(A\) an object of \(\cat{E}_{I}\).
Write \(1_{I}\) for the terminal object of \(\cat{E}_{I}\).
The unique morphism \(f : 1_{I} \to 1_{J}\) above \(u\)
is a cartesian morphism and thus a fibration.
Define \(s_{*}A = \Pi_{f}A\).
Since \(p : \cat{E} \to \cat{B}\) preserves
dependent products, \(s_{*}A\) is above \(J\).
Therefore \(s_{*}\) determines a functor \(\cat{E}_{I} \to \cat{E}_{J}\)
which is a right adjoint to \(s^{*}\)
and satisfies the Beck-Chevalley condition
by Lemma \ref{orgf15694c}.
\end{proof}

\subsection{Internal Language for a Fibred Type-Theoretic Fibration Category}
\label{sec:org8daeffb}
We define the \emph{internal language} \(\lang{p}\)
for a fibred type-theoretic fibration category \(p : \cat{E} \to \cat{B}\).
It is a type theory with two sorts \emph{kind} and \emph{type},
where types depend on some kinds.
The theory of kinds is the internal language of \(\cat{B}\)
which is a Martin-L\"{o}f type theory
written in the manner
\(\gamma : \Gamma \vdash \Delta(\gamma) \ \kind\)
for a kind judgment and
\(\gamma : \Gamma \vdash \delta(\gamma) : \Delta(\gamma)\)
for a term judgment.
The theory of types over a kind \(\Gamma\)
is the internal language of \(\cat{E}_{\Gamma}\),
written in the manner
\(\gamma : \Gamma \mid a : A(\gamma) \vdash B(\gamma; a) \ \type\)
for a type judgment and
\(\gamma : \Gamma \mid a : A(\gamma) \vdash b(\gamma; a) : B(\gamma; a)\)
for a term judgment.
Corresponding to reindexings,
there are rules for \emph{substitution}:
\begin{prooftree}
\AxiomC{\(\Gamma \vdash f : \Delta\)}
\AxiomC{\(\delta : \Delta \mid a : A(\delta) \vdash B(\delta; a) \ \type\)}
\BinaryInfC{\(\Gamma \mid a : A(f) \vdash B(f; a) \ \type\)}
\end{prooftree}
and
\begin{prooftree}
\AxiomC{\(\Gamma \vdash f : \Delta\)}
\AxiomC{\(\delta : \Delta \mid a : A(\delta) \vdash b(\delta; a) : B(\delta; a)\)}
\BinaryInfC{\(\Gamma \mid a : A(f) \vdash b(f; a) : B(f; a)\).}
\end{prooftree}
Corresponding to the condition \ref{org45c198d} of Theorem \ref{org4e7cd71},
there is a \emph{dependent product type over a kind}:
\begin{prooftree}
\AxiomC{\(\Gamma \vdash \Delta \ \kind\)}
\AxiomC{\(\Gamma, \delta : \Delta \mid a : A(\delta) \vdash B(\delta; a) \ \type\)}
\BinaryInfC{\(\Gamma \mid a : \Pi_{\delta : \Delta}A(\delta) \vdash \Pi_{\delta : \Delta}B(\delta; a\delta) \ \type\)}
\end{prooftree}
with the obvious introduction and elimination rules.
The condition \ref{orgd4716d9} of Theorem \ref{org4e7cd71}
corresponds to the \emph{path induction on an identity kind with respect to type families}:
\begin{prooftree}
\AxiomC{\(\gamma : \Gamma, \gamma' : \Gamma, \delta : \gamma = \gamma' \mid a : A(\delta) \vdash B(\delta; a) \ \type\)}
\noLine
\UnaryInfC{\(\gamma : \Gamma \mid a : A(\refl_{\gamma}) \vdash b(\gamma; a) : B(\refl_{\gamma}; a)\)}
\UnaryInfC{\(\gamma : \Gamma, \gamma' : \Gamma, \delta : \gamma = \gamma' \mid a : A(\delta) \vdash \Ind_{=_{\Gamma}}^{B, b}(\delta; a) : B(\delta; a)\)}
\end{prooftree}
with the computational rule
\(\Ind_{=_{\Gamma}}^{B, b}(\refl_{\gamma}; a) \equiv b(\gamma; a)\).

Theorem \ref{org4e7cd71} means that
the total category is a model of Martin-L\"{o}f type theory
where types are pairs of \(() \vdash \Gamma \ \kind\)
and \(\gamma : \Gamma \mid () \vdash A(\gamma) \ \type\).
The construction of factorization described in the proof of Lemma \ref{org1b5b6ac}
implies that the identity type of
\((() \vdash \Gamma \ \kind, \gamma : \Gamma \mid () \vdash a : A(\gamma) \ \type)\)
is implemented as
\((\gamma : \Gamma, \gamma' : \Gamma \vdash \gamma = \gamma \ \kind,
\gamma : \Gamma, \gamma' : \Gamma, \delta : \gamma = \gamma
\mid a : A(\gamma), a' : A(\gamma') \vdash \delta_{*}a = a' \ \type)\),
where \(\delta_{*} : A(\gamma) \to A(\gamma')\)
is the transport along the path \(\delta\).
The type \(\delta_{*}a = a'\) is a type of \emph{paths from \(a\) to \(a'\) over \(\delta\)}
and written as \(a =_{\delta} a'\).
The construction of dependent product described in the proof of Lemma \ref{org1e0b166}
implies that the dependent product of a type family
\((\gamma : \Gamma \vdash \Delta(\gamma) \ \kind,
\gamma : \Gamma, \delta : \Delta(\gamma) \mid a : A(\gamma) \vdash B(\gamma, \delta; a))\)
over \((() \vdash \Gamma \ \kind, \gamma : \Gamma \mid () \vdash A(\gamma) \ \type)\)
is implemented as
\((() \vdash \Pi_{\gamma : \Gamma}\Delta(\gamma) \ \kind,
f : \Pi_{\gamma : \Gamma}\Delta(\gamma) \mid () \vdash
\Pi_{\gamma : \Gamma}\Pi_{a : A(\gamma)}B(\gamma, f\gamma; a) \ \type)\).

\label{sec-4-5-1}
As a consequence of the soundness of the interpretation of Martin-L\"{o}f type theory
in type-theoretic fibration categories,
we have an important property of logical predicates
so called the ``basic lemma.''
\begin{definition}
\label{sec-4-5-1-1}
For a Martin-L\"{o}f type theory \(M\),
write \(\cat{T}(M)\) for the syntactic category of \(M\).
For a type-theoretic fibration category \(\cat{C}\),
an \emph{interpretation of \(M\) in \(\cat{C}\)}
is a type-theoretic functor from \(\cat{T}(M)\) to \(\cat{C}\).
\end{definition}
\begin{definition}
\label{sec-4-5-1-2}
Let \(p : \cat{E} \to \cat{B}\) be a fibred type-theoretic fibration category
and \(F : \cat{T}(M) \to \cat{B}\)
an interpretation of a Martin-L\"{o}f type theory \(M\) in \(\cat{B}\).
A \emph{logical predicate on \(F\) with respect to \(p\)}
is an interpretation \(R : \cat{T}(M) \to \cat{E}\)
such that \(p \circ R = F\).
\end{definition}
\begin{corollary}[{Basic Lemma}]
\label{sec-4-5-1-3}
\label{orgb94bb63}
Let \(p : \cat{E} \to \cat{B}\) be a fibred type-theoretic fibration category
and \(R : \cat{T}(M) \to \cat{E}\) a logical predicate
on an interpretation \(F : \cat{T}(M) \to \cat{B}\)
of a Martin-L\"{o}f type theory \(M\) in \(\cat{B}\).
Then for any term \(a : A \vdash t(a) : B(a)\),
there exists a term \(a : FA \mid r : RA(a) \vdash \hat{t}(a; r) : RB(a, Ft(a); r)\)
in the internal language for \(p\),
where \(\hat{t}\) is the induced morphism
\(RA \to (Ft)^{*}RB\) from \(Rt : RA \to RB\).
\end{corollary}

\section{Univalence in a Fibred Type-Theoretic Fibration Category}
\label{sec:org9fcc347}
\label{orgd199e01}
We construct a univalent universe in the total category
of a fibred type-theoretic fibration category
from univalent universes in the base category
and in the fiber at the terminal object.
The new universe is fibred in some sense
and called a \emph{fibred universe}.
A fibred univalent universe is preserved by the change of base
along a functor preserving small fibrations.
\subsection{Universes in a Fibred Setting}
\label{sec:org0e5a49a}

\begin{definition}
\label{sec-5-1-1}
\label{org6a6ca22}
\cite[Definition 6.12]{shulman2015inverse}
A fibration \(u : \tilde{U} \twoheadrightarrow U\)
in a type-theoretic fibration category \(\cat{C}\)
is a \emph{universe} if the following conditions hold,
where ``(\(u\)-)small fibration'' means ``a pullback of \(u\)''.
\begin{enumerate}
\item \label{org33bfa1a}
Small fibrations are closed under composition
and contain the identities.
\item \label{org9cad704}
If \(f : A \twoheadrightarrow B\) and \(g : B \twoheadrightarrow C\)
are small fibrations, so is \(\Pi_{g}f : \Pi_{g}A \twoheadrightarrow C\).
\item \label{orga34caef}
If \(A \twoheadrightarrow C\) and \(B \twoheadrightarrow C\)
are small fibrations, then
any morphism \(f : A \to B\) over \(C\)
factors as an acyclic cofibration followed by a small fibration.
\end{enumerate}
\end{definition}

\begin{definition}
\label{sec-5-1-2}
\label{org2fc2ff6}
Let \(p : \cat{E} \to \cat{B}\)
be a fibred type-theoretic fibration category.
A \emph{fibred universe} in \(p\)
is a fibration \(u : \tilde{U} \twoheadrightarrow U\) in \(\cat{E}\)
satisfying the following conditions.
\begin{enumerate}
\item Both \(u\) and \(pu\) are universes
in \(\cat{E}\) and \(\cat{B}\) respectively.
\item \label{org6476321}
For a \(u\)-small fibration \(f : A \twoheadrightarrow B\)
and a pullback
\[
\begin{tikzcd}
 pA \arrow[r,"k"] \arrow[d,twoheadrightarrow] & p\tilde{U} \arrow[d,twoheadrightarrow] \\
 pB \arrow[r,"h"'] & pU, \end{tikzcd}
\]
there exists a pullback
\[
\begin{tikzcd}
 A \arrow[r,"\bar{k}"] \arrow[d,twoheadrightarrow] & \tilde{U} \arrow[d,twoheadrightarrow] \\
 B \arrow[r,"\bar{h}"'] & U \end{tikzcd}
\]
above \(h\) and \(k\).
\item \label{org62afa58}
Every cartesian morphism above a \(pu\)-small fibration is
a \(u\)-small fibration.
\item \label{orgaf7f6ca}
For every \(u\)-small fibration \(g : A \twoheadrightarrow C\),
\(pu\)-small fibration \(s : pA \twoheadrightarrow J\)
and arbitrary morphism \(t : J \to pC\)
such that \(s \circ t = pg\),
the induced morphism \(A \to t^{*}C\)
is a \(u\)-small fibration.
\end{enumerate}
\end{definition}
\begin{proposition}
\label{sec-5-1-3}
\label{orgc1f21e5}
Let \(p : \cat{E} \to \cat{B}\)
be a fibred type-theoretic fibration category.
Let \(u : \tilde{U} \twoheadrightarrow U\)
be a fibration in \(\cat{B}\) and
\(v : \tilde{V} \twoheadrightarrow V\)
a fibration in \(\cat{E}_{1}\).
Suppose the following conditions hold,
where \(v_{I} : \tilde{V}_{I} \to V_{I}\)
is the reindexing of \(v\) along
the unique morphism \(I \to 1\) for \(I \in \cat{B}\).
\begin{enumerate}
\item The fibrations \(u\) and all \(v_{I}\) are universes.
\item The fibration \(u_{*}v_{\tilde{U}} : u_{*}\tilde{V}_{\tilde{U}} \twoheadrightarrow u_{*}V_{\tilde{U}}\)
is small in \(\cat{E}_{U}\).
\end{enumerate}
Then \(p\) has a fibred universe \(w : \tilde{W} \twoheadrightarrow W\)
above \(u\), defined as \(W = u_{*}V_{\tilde{U}}\) and
\(\tilde{W} = \varepsilon^{*}\tilde{V}_{\tilde{U}}\)
where \(\varepsilon : u^{*}u_{*}V_{\tilde{U}} \to V_{\tilde{U}}\)
is the counit of the adjunction \(u^{*} \dashv u_{*}\).
\[
\begin{tikzcd}
 [row sep=3ex] \tilde{V}_{\tilde{U}}  \arrow[d,twoheadrightarrow] & \varepsilon^{*}\tilde{V}_{\tilde{U}}  = \tilde{W} \arrow[l] \arrow[d,twoheadrightarrow] \arrow[dr,twoheadrightarrow,"w"] &  \\
 V_{\tilde{U}} & u^{*}u_{*}V_{\tilde{U}}  \arrow[l,"\varepsilon"'] \arrow[r] & u_{*}V_{\tilde{U}}  = W \\
 \tilde{U} & \tilde{U} \arrow[r,twoheadrightarrow,"u"'] \arrow[l,equal] & U \end{tikzcd}
\]
\end{proposition}

\label{sec-5-1-4}
To prove Proposition \ref{orgc1f21e5},
we first characterize the \(w\)-small fibrations.
\begin{lemma}
\label{sec-5-1-4-1}
\label{orga29b979}
In the setting of Proposition \ref{orgc1f21e5},
a fibration \(f : A \twoheadrightarrow B\) in \(\cat{E}\)
is \(w\)-small if and only if
\(pf\) is \(u\)-small and
the induced morphism \(A \twoheadrightarrow (pf)^{*}B\)
is \(v_{pX}\)-small.
\end{lemma}
\begin{proof}
\label{sec-5-1-4-2}
Observe that a pullback of \(w\) along \(g : B \to W\)
is calculated by:
\begin{enumerate}
\item pulling \(\tilde{U}\) back along \(pg\) as
\[
\begin{tikzcd}
 I \arrow[r,"k"] \arrow[d,twoheadrightarrow,"s"] & \tilde{U} \arrow[d,twoheadrightarrow,"u"] \\
 pB \arrow[r,"pg"'] & U, \end{tikzcd}
\]
and
\item calculating a pullback square in \(\cat{E}_{I}\)
\[
\begin{tikzcd}
 g^{*}W \arrow[rr] \arrow[d,twoheadrightarrow] &  & k^{*}\tilde{W} \arrow[d,twoheadrightarrow] \\
 s^{*}B \arrow[r] & s^{*}(pg)^{*}W \arrow[r,"\cong"] & k^{*}u^{*}W. \end{tikzcd}
\]
\end{enumerate}
The second pullback can be extended to pullbacks in \(\cat{E}_{I}\)
\[
\begin{tikzcd}
 g^{*}W \arrow[d,twoheadrightarrow] \arrow[r] & k^{*} \varepsilon^{*}\tilde{V}_{\tilde{U}}  \arrow[d,twoheadrightarrow] \arrow[r] & k^{*}\tilde{V}_{\tilde{U}}  \arrow[d,twoheadrightarrow] \arrow[r,"\cong"] & \tilde{V}_{I}  \arrow[d,twoheadrightarrow,"v_{I}"] \\
 s^{*}B \arrow[r] & k^{*}u^{*}u_{*}V_{\tilde{U}}  \arrow[r,"k^{*} \varepsilon"'] & k^{*}V_{\tilde{U}}  \arrow[r,"\cong"] & V_{I}. \end{tikzcd}
\]
Therefore if \(f : A \twoheadrightarrow B\) is a pullback of \(w\),
then \(pf\) is a pullback of \(u\) and the induced morphism \(A \twoheadrightarrow (pf)^{*}B\)
is a pullback of \(v_{pA}\).
Conversely, suppose \(pf\) is a pullback of \(u\) along \(t : pB \to U\)
with upper morphism \(k : pA \to \tilde{U}\),
and \(A \twoheadrightarrow (pf)^{*}B\) is a pullback of
\(v_{pA}\) along \(h : (pf)^{*}B \to V_{pA}\) above \(pA\).
Then there exists a unique morphism \(\bar{h} : B \to (pf)_{*}V_{pA}\) above \(pB\)
such that \(\varepsilon' \circ (pf)^{*}\bar{h} = h\)
where \(\varepsilon' : (pf)^{*}(pf)_{*} \Rightarrow 1\)
is the counit of the adjunction \((pf)^{*} \dashv (pf)_{*}\).
But \((pf)^{*}(pf)_{*}V_{pA} \cong (pf)^{*}(pf)_{*}k^{*}V_{\tilde{U}}
\cong (pf)^{*}t^{*}u_{*}V_{\tilde{U}} \cong k^{*}u^{*}u_{*}V_{\tilde{U}}
= k^{*}u^{*}W\),
and thus \(A\) is a pullback of \(k^{*}\tilde{W} \twoheadrightarrow k^{*}u^{*}W\)
along \((pf)^{*}\bar{h} : (pf)_{*}B \to k^{*}u^{*}W\).
Hence \(f : A \twoheadrightarrow B\) is a pullback of \(w\).
\end{proof}

\begin{proof}[{Proof of Proposition \ref{orgc1f21e5}}]
\label{sec-5-1-4-3}
We check that \(w\) satisfies the conditions of Definition \ref{org6a6ca22}.
Using Lemma \ref{orga29b979},
the condition \ref{org33bfa1a} is clear.
The condition \ref{orga34caef} follows from the construction of factorization in \(\cat{E}\)
given in Lemma \ref{org1b5b6ac}.
By the construction of dependent products in \(\cat{E}\)
given in Lemma \ref{org1e0b166},
in order to show \ref{org9cad704}
it is enough to prove that
for every \(u\)-small fibration \(f : I \twoheadrightarrow J\),
\(f_{*}\) preserves small fibrations.
To see this it is enough to show that
\(f_{*}v_{I} : f_{*}\tilde{V}_{I} \twoheadrightarrow f_{*}V_{I}\)
is a \(v_{J}\)-small fibration,
because \(f_{*}\) preserves pullbacks.
Suppose \(f\) is a pullback of \(u\) along \(k : J \to U\)
with upper morphism \(h : I \to \tilde{U}\).
Then \(f_{*}v_{I} \cong f_{*}h^{*}v_{\tilde{U}}
\cong k^{*}u_{*}v_{\tilde{U}}\) by the Beck-Chevalley condition.
Now \(u_{*}v_{\tilde{U}}\) is a small fibration by assumption,
and thus so is \(f_{*}v_{I}\).

It is easy to show that \(w\) is a fibred universe
using Lemma \ref{orga29b979}.
\end{proof}

\label{sec-5-1-5}
The change of base along a suitable functor
creates a new fibred universe.
\begin{proposition}
\label{sec-5-1-5-1}
\label{org4a92ba1}
Let \(p : \cat{E} \to \cat{B}\)
be a fibred type-theoretic fibration category
with a fibred universe \(w : \tilde{W} \twoheadrightarrow W\)
above \(u : \tilde{U} \twoheadrightarrow U\),
\(\cat{A}\) a type-theoretic fibration category
with a universe \(v : \tilde{V} \twoheadrightarrow V\),
and \(F : \cat{A} \to \cat{B}\) a functor preserving
fibrations, pullbacks of fibrations and acyclic cofibrations.
Suppose \(Fv\) is a pullback of \(u\) along a morphism \(h : FV \to U\)
with upper morphism \(k : F\tilde{V} \to \tilde{U}\).
Then \((v, h^{*}w) : (\tilde{V}, k^{*}\tilde{W}) \twoheadrightarrow (V, h^{*}W)\)
is a fibred universe in \(F^{*}\cat{E}\).
\end{proposition}
\begin{proof}
\label{sec-5-1-5-2}
First we show that a fibration \((s, f) : (I, A) \twoheadrightarrow (J, B)\)
is \((v, h^{*}w)\)-small if and only if
\(s\) is \(v\)-small and \(f\) is \(w\)-small.
The ``only if'' part is trivial.
To show the converse, suppose \(s\) is \(v\)-small and \(f\) is \(w\)-small.
Then \(s\) is a pullback of \(v\) along some morphism \(t : J \to V\).
Since \(f\) is a \(w\)-small fibration and
\(pf = Fs\) is a pullback of \(u\) along \(h \circ Ft\),
\(f\) is a pullback of \(w\) along some morphism \(k\) above \(h \circ Ft\)
by the condition \ref{org6476321} of Definition \ref{org2fc2ff6}.
Therefore \(f\) is a pullback of \(h^{*}w\)
along the induced morphisms \(B \to h^{*}W\) above \(Ft\),
and this means that \((s, f)\) is a pullback of \((v, h^{*}w)\) in \(F^{*}\cat{E}\).

We show that \((v, h^{*}w)\) is a universe in \(F^{*}\cat{E}\).
The conditions \ref{org33bfa1a} and \ref{orga34caef} of Definition \ref{org6a6ca22}
follows from the above characterization of \((v, h^{*}w)\)-small fibrations.
To show the condition \ref{org9cad704},
let \((s, f) : (I, A) \twoheadrightarrow (J, B)\)
and \((t, g) : (J, B) \twoheadrightarrow (K, C)\)
be \((v, h^{*}w)\)-small fibrations.
By the construction of dependent products in \(F^{*}\cat{E}\)
described in the proof of Proposition \ref{orgc8a114a},
it is enough to show that
\(\varepsilon'^{*}\Pi_{g}A \twoheadrightarrow C\)
is a \(w\)-small fibration,
where \(\varepsilon' : F(\Pi_{t}I) \to \Pi_{Ft}FI\)
is the canonical morphism.
This fibration factors as
\[
\begin{tikzcd}
 \varepsilon'^{*} \Pi_{g}A \arrow[r] \arrow[d,twoheadrightarrow,"h'"'] & \Pi_{g}A \arrow[dr,twoheadrightarrow,"\Pi_{g}f"] \arrow[d,twoheadrightarrow,"h"] &  \\
 (F(\Pi_{t}s))^{*}C \arrow[r] \arrow[rr,twoheadrightarrow,bend right=15,"l"'] & (\Pi_{Ft}Fs)^{*}C \arrow[r,"k"] & C. \end{tikzcd}
\]
Since \(\Pi_{g}f\), \(F(\Pi_{t}s)\) and \(\Pi_{Ft}Fs\)
are small fibrations,
\(k\) and \(l\) are \(w\)-small by the condition \ref{org62afa58} of Definition \ref{org2fc2ff6},
and \(h\) is \(w\)-small by the condition \ref{orgaf7f6ca}.
The left square is a pullback,
and thus \(h'\) is a \(w\)-small fibration and so is \(\varepsilon'^{*}\Pi_{g}A \twoheadrightarrow C\).

It is clear that the new universe \((v, h^{*}w)\)
is a fibred universe in \(F^{*}\cat{E} \to \cat{A}\).
\end{proof}

\subsection{Univalence in a Fibred Setting}
\label{sec:org1581480}
For a fibration \(u : \tilde{U} \twoheadrightarrow U\),
write \(E(u) \twoheadrightarrow U \times U\)
for the fibration corresponding to the type
\(a : U, b : U \vdash \tilde{U}(a) \simeq \tilde{U}(b)\),
where \(A \simeq B \equiv
\Sigma_{f : A \to B}(\Sigma_{g : B \to A}\Pi_{x : A}g(fx) = x)
\times (\Sigma_{h : B \to A}\Pi_{y : B}f(hy) = y)\)
is the type of bi-invertible maps.
The object \(E(u)\) has the following universal property:
for a morphism \(\langle a, b \rangle : X \to U \times U\),
there is a natural one-to-one correspondence between
the set of the morphisms \(X \to E(u)\) over \(U \times U\) and
the set of the quintuples
\((f : a^{*}\tilde{U} \to b^{*}\tilde{U},
g : b^{*}\tilde{U} \to a^{*}\tilde{U}, \sigma : a^{*}\tilde{U} \to P_{U}\tilde{U},
h : b^{*}\tilde{U} \to a^{*}\tilde{U}, \tau : b^{*}\tilde{U} \to P_{U}\tilde{U})\)
such that
\(f\), \(g\) and \(h\) are over \(X\)
and the following diagrams commute
\[
\begin{tikzcd}
 [column sep=2ex] a^{*}\tilde{U} \arrow[r,"\sigma"] \arrow[d,"\pair{gf}{1}"'] & P_{U}\tilde{U} \arrow[d,twoheadrightarrow] & b^{*}\tilde{U} \arrow[r,"\tau"] \arrow[d,"\pair{fh}{1}"'] & P_{U}\tilde{U} \arrow[d,twoheadrightarrow] \\
 a^{*}\tilde{U} \times_{X}  a^{*}\tilde{U} \arrow[r] & \tilde{U} \times_{U}  \tilde{U} & b^{*}\tilde{U} \times_{X}  b^{*}\tilde{U} \arrow[r] & \tilde{U} \times_{U}  \tilde{U}. \end{tikzcd}
\]
Note that this definition depends on the choice of path object \(P_{U}\tilde{U}\),
and we assume that every fibration has a fixed path object
in the rest of this section.
There are canonical morphisms
\((f_{u} : \pi_{1}^{*}\tilde{U} \to \pi_{2}^{*}\tilde{U}, g_{u} : \pi_{2}^{*}\tilde{U} \to \pi_{1}^{*}\tilde{U},
\sigma_{u} : \pi_{1}^{*}\tilde{U} \to P_{U}\tilde{U}, h_{u} : \pi_{2}^{*}\tilde{U} \to \pi_{1}^{*}\tilde{U},
\tau_{u} : \pi_{2}^{*}\tilde{U} \to P_{U}\tilde{U})\)
corresponding to the identity \(E(u) \to E(u)\),
where \(\pi_{1}, \pi_{2} : E(u) \to U\) are projections.
There is a canonical morphism
\(e(u) : U \to E(u)\) over the diagonal morphism \(U \to U \times U\)
which corresponds to the identity function.
\begin{definition}
\label{sec-5-2-1}
A fibration \(u : \tilde{U} \twoheadrightarrow U\) is \emph{univalent} if
the canonical morphism \(e(u) : U \to E(u)\) is a homotopy equivalence.
\end{definition}
\begin{lemma}
\label{sec-5-2-2}
\label{org30c30d1}
In a fibred type-theoretic fibration category \(\cat{E} \to \cat{B}\),
every cartesian morphism above a homotopy equivalence
is a homotopy equivalence.
\end{lemma}
\begin{proof}
\label{sec-5-2-3}
We show that a cartesian morphism above a half adjoint equivalence
is a homotopy equivalence.
Suppose \(f : I \to J\) is a half adjoint equivalence in \(\cat{B}\)
with \(g : J \to I\), \(\eta : gf \sim 1\) and \(\varepsilon : fg \sim 1\)
and \(Y \in \cat{E}_{J}\).
We construct a homotopy inverse \(\bar{g}\)
of the cartesian morphism \(\bar{f} : f^{*}Y \to Y\).
In the internal language,
\(\bar{f}\) is the identity
\(i : I \mid y : Y(fi) \vdash y : Y(fi)\) and
\(\bar{g}\) is a term of type
\(j : J \mid y : Y(j) \vdash \bar{g}(j; y) : Y(f(gj))\).
We set \(\bar{g}(j; y) \equiv \varepsilon_{j}^{*}y\),
a backward transport of \(y\) along the path \(\varepsilon_{j} : f(gj) = j\).
Then there exists a homotopy
\(j : J \mid y : Y(j) \vdash \bar{\varepsilon} : \bar{g}(j; y) =_{\varepsilon_{j}} y\).
Since \(f\eta \sim_{J \times J} \varepsilon f\), there exists a homotopy
\(i : I \mid y : Y(fi) \vdash \bar{\eta} : \bar{g}(fi; y) =_{f\eta_{i}} y\).
Hence \(\bar{g}\) is a homotopy inverse of \(f\).
\end{proof}

\begin{lemma}
\label{sec-5-2-4}
\label{org899ecae}
Let \(p : \cat{E} \to \cat{B}\)
be a fibred type-theoretic fibration category.
For a morphism \(f : X \to Y\) in \(\cat{E}\),
\(f\) is a homotopy equivalence if and only if
\(pf\) and the induced morphism \(X \to (pf)^{*}Y\)
are homotopy equivalences in \(\cat{B}\) and \(\cat{E}_{pX}\) respectively.
\end{lemma}

\begin{proof}
\label{sec-5-2-5}
The ``if'' part is a corollary of Lemma \ref{org30c30d1}.
To show the ``only if'' part,
let \(f : X \to Y\) be a half adjoint equivalence in \(\cat{E}\)
with \(g : Y \to X\), \(\eta : gf \sim 1\) and \(\varepsilon : fg \sim 1\).
Let \(s = pf : I \to J\) and \(t = pg\), \(\sigma = p\eta\) and \(\tau = p\varepsilon\)
which make \(s\) a half adjoint equivalence in \(\cat{B}\).
We construct a homotopy inverse \(\bar{g} : s^{*}Y \to X\)
of the induced morphism \(\bar{f} : X \to s^{*}Y\) in \(\cat{E}_{pX}\).
In the internal language,
\(\bar{f}\) is \(f\) itself \(i : I \mid x : X(i) \vdash f(i; x) : Y(si)\)
and \(g\) is a term of type \(j : J \mid y : Y(j) \vdash g(j; y) : X(tj)\).
Let \(\bar{g}\) be the term
\(i : I \mid y : Y(si) \vdash (\sigma_{i})_{*}g(si; y) : X(i)\).
Then \(i : I \mid x : X(i) \vdash (\bar{g}(si; f(i; x)) : X(i)\)
is homotopic to \(x\) via \(\eta_{x} : g(si; f(i; x)) =_{\sigma_{i}} x\).
To give a homotopy \(\bar{f}\bar{g} \sim 1\),
let \(i : I\) and \(y : Y(si)\).
By definition there is a path \(\bar{\sigma_{i}} : gy =_{\sigma_{i}} \bar{g}y\).
Applying \(f\) we have a path \(f(gy) =_{s\sigma_{i}} \bar{f}(\bar{g}y)\).
Also we have paths \(\varepsilon_{y} : f(gy) =_{\tau_{si}} y\)
and \(s\sigma_{i} = \tau_{si}\) by assumption.
Thus there exists a path \(\bar{f}(\bar{g}y) = y\) in \(Y(i)\).
Hence \(\bar{g}\) is a homotopy inverse of \(\bar{f}\) in \(\cat{E}_{pX}\).
\end{proof}

\begin{lemma}
\label{sec-5-2-6}
\label{orgca61317}
In the following diagram in a fibred type-theoretic fibration category
\(\cat{E} \to \cat{B}\)
\[
\begin{tikzcd}
 [row sep=2ex] A' \arrow[rr] \arrow[dr,"f'"'] &  & A \arrow[dr,"f"] &  \\
  & B' \arrow[rr] &  & B \\
 I' \arrow[rr] \arrow[dr,"s'"'] &  & I \arrow[dr,"s"] &  \\
  & J' \arrow[rr] &  & J, \end{tikzcd}
\]
if the horizontal morphisms are cartesian and
\(s'\) and \(f\) are homotopy equivalences,
then \(f'\) is a homotopy equivalence.
\end{lemma}
\begin{proof}
\label{sec-5-2-7}
By Lemma \ref{org899ecae}, it is enough to show that
the induced morphism \(\bar{f'} : A' \to s'^{*}B'\)
is a homotopy equivalence in \(\cat{E}_{I'}\).
The induced morphism \(\bar{f} : A \to s^{*}B\)
is a homotopy equivalence by Lemma \ref{org899ecae}.
The morphism \(\bar{f'}\) is the image of \(\bar{f}\)
by the reindexing functor along \(I' \to I\),
and thus \(\bar{f'}\) is a homotopy equivalence by Lemma \ref{orgb900b60}.
\end{proof}
\begin{proposition}
\label{sec-5-2-8}
\label{org8640dba}
Let \(p : \cat{E} \to \cat{B}\)
be a fibred type-theoretic fibration category,
\(u : \tilde{U} \twoheadrightarrow U\) a fibration in \(\cat{B}\)
and \(v : \tilde{V} \twoheadrightarrow V\) a fibration in \(\cat{E}_{1}\).
Suppose \(u\) and \(v\) are univalent
and \(u_{*}\) preserves homotopy equivalences.
Then the fibration \(w\) in Proposition \ref{orgc1f21e5} is univalent.
\end{proposition}
\begin{proof}
\label{sec-5-2-9}
We show that the canonical morphism \(e(w) : W \to E(w)\)
is a homotopy equivalence.
By Lemma \ref{org899ecae},
it is enough to show that the canonical morphism \(e(u) : U \to E(u)\)
and the induced morphism \(W \to e(u)^{*}E(w)\)
are homotopy equivalences in \(\cat{B}\) and \(\cat{E}_{U}\) respectively.
The morphism \(e(u)\) is a homotopy equivalence by assumption.
The object \(e(u)^{*}E(w)\) corresponds to the type
\(a : U \mid b : \Pi_{s : \tilde{U}(a)}V, b' : \Pi_{s : \tilde{U}(a)}V
\vdash \Pi_{s : \tilde{U}(a)}\tilde{V}(bs) \simeq \tilde{V}(b's)\),
and this type also corresponds to the object
\(u_{*}!_{\tilde{U}}^{*}E(v)\) where
\(!_{\tilde{U}}\) is the unique arrow \(\tilde{U} \to1\).
Thus the morphism \(W \to e(u)^{*}E(w)\)
is isomorphic to
\(u_{*}!_{\tilde{U}}^{*}e(v) : W \to u_{*}!_{\tilde{U}}^{*}E(v)\)
along \(e(u)^{*}E(w) \cong u_{*}!_{\tilde{U}}^{*}E(v)\).
The latter morphism \(u_{*}!_{U}^{*}e(v)\) is a homotopy equivalence because
\(v\) is univalent,
the reindexing functor preserves homotopy equivalences,
and so does \(u_{*}\) by assumption.
\end{proof}
\begin{lemma}
\label{sec-5-2-10}
\label{org51efab6}
Let
\[
\begin{tikzcd}
 \tilde{U'} \arrow[d,twoheadrightarrow,"u'"'] \arrow[r] & \tilde{U} \arrow[d,twoheadrightarrow,"u"] \\
 U' \arrow[r,"f"'] & U \end{tikzcd}
\]
be a pullback square in a type-theoretic fibration category.
Then
\[
\begin{tikzcd}
 E(u') \arrow[d,twoheadrightarrow] \arrow[r] & E(u) \arrow[d,twoheadrightarrow] \\
 U' \times U' \arrow[r,"f \times f"'] & U \times U \end{tikzcd}
\]
is a pullback square.
\end{lemma}
\begin{proof}
\label{sec-5-2-11}
The fibration \(E(u') \twoheadrightarrow U' \times U'\)
corresponds to the type \(x : U', y : U' \vdash \tilde{U'}(x) \simeq \tilde{U'}(y)\).
But \(\tilde{U'}(x) \simeq \tilde{U'}(y) \equiv \tilde{U}(fx) \simeq \tilde{U}(fy)\),
which corresponds to a pullback of \(E(u)\) along \(f \times f\).
\end{proof}
\begin{proposition}
\label{sec-5-2-12}
\label{orgc52c97e}
Let \(p : \cat{E} \to \cat{B}\)
be a fibred type-theoretic fibration category
with a univalent fibration \(w : \tilde{W} \twoheadrightarrow W\)
above \(u : \tilde{U} \twoheadrightarrow U\),
\(\cat{A}\) a type-theoretic fibration category
with a univalent fibration \(v : \tilde{V} \twoheadrightarrow V\),
and \(F : \cat{A} \to \cat{B}\)
a functor preserving fibrations, pullbacks of fibrations
and acyclic cofibrations.
Suppose \(Fv\) is a pullback of \(u\) along a morphism
\(h : FV \to U\) with upper morphism \(k : FV \to \tilde{U}\).
Then \((v, h^{*}w) : (\tilde{V}, k^{*}W) \twoheadrightarrow (V, h^{*}W)\)
is a univalent fibration in \(F^{*}\cat{E}\).
\end{proposition}
\begin{proof}
\label{sec-5-2-13}
We first describe the canonical morphism
\(e(v, h^{*}w) : (V, h^{*}W) \to E(v, h^{*}w)\) in \(F^{*}\cat{E}\).
There is a canonical morphism \(c : F(E(v)) \to E(Fv)\)
corresponding to \((F(f_{v}), F(g_{v}), F(\sigma_{v}), F(h_{v}), F(\tau_{v}))\),
where we choose \(F(P_{V}\tilde{V})\) as a path object of \(F(v)\).
It is easy to show that \(E(v, h^{*}w)\) is a reindexing of \(E(h^{*}w)\) along \(c\)
and \(e(v, h^{*}w)\) is the induced morphism \(f : h^{*}W \to c^{*}E(h^{*}w)\)
\[
\begin{tikzcd}
 [row sep=2ex] h^{*}W \arrow[drr,"e(h^{*}w)"] \arrow[dr,dashrightarrow,"f"'] &  &  \\
  & c^{*}E(h^{*}w) \arrow[r] & E(h^{*}w) \\
 FV \arrow[drr,"e(Fv)"] \arrow[dr,"F(e(v))"'] &  &  \\
  & F(E(v)) \arrow[r,"c"'] & E(Fv) \end{tikzcd}
\]
checking the universal property.

We have to show that \(f : h^{*}W \to c^{*}E(h^{*}w)\) is a homotopy equivalence.
By Lemma \ref{org51efab6},
\(E(h^{*}w)\) is a pullback of \(E(w)\) along the morphism
\(\bar{h} \times \bar{h} : h^{*}W \times h^{*}W \to W \times W\).
The morphism \(\bar{h} \times \bar{h}\) is a cartesian morphism
and so is the upper morphism \(E(h^{*}w) \to E(w)\).
Hence in the following diagram in \(\cat{E} \to \cat{B}\)
\[
\begin{tikzcd}
 [row sep=2ex] h^{*}W \arrow[rr] \arrow[dr,"f"'] &  & W \arrow[dr,"e(w)"] &  \\
  & c^{*}E(h^{*}w) \arrow[rr] &  & E(w) \\
 FV \arrow[rr] \arrow[dr,"F(e(v))"'] &  & U \arrow[dr,"e(u)"] &  \\
  & F(E(v)) \arrow[rr] &  & E(u), \end{tikzcd}
\]
the horizontal morphisms are cartesian,
and \(F(e(v))\) and \(e(w)\) are homotopy equivalences.
Thus \(f\) is a homotopy equivalence by Lemma \ref{orgca61317}.
\end{proof}

\begin{example}
\label{sec-5-2-14}
\label{org0da0785}
Let \(\cat{C}\) be a type-theoretic fibration category
with a univalent universe \(u : \tilde{U} \to U\).
Consider the codomain functor \(\cod : \fibcat{\arcat{\cat{C}}} \to \cat{C}\).
Its fiber at \(1\) is \(\cat{C}\) which has a univalent universe \(u\).
The fiber at an object \(A\) is \(\fibcat{\cat{C}/A}\)
whose type-theoretic structure is inherited from \(\cat{C}\).
Thus each \(A \times u\) is a universe in \(\fibcat{\cat{C}/A}\).
Since \(u_{*} : \fibcat{\cat{C}/\tilde{U}} \to \fibcat{\cat{C}/U}\)
is given by dependent products,
it preserves small fibrations.
Hence the codomain functor has a fibred universe above \(u\)
by Proposition \ref{orgc1f21e5}.
Since the univalence axiom implies the function extensionality,
\(u_{*}\) preserves homotopy equivalences.
Thus this fibred universe is univalent by Proposition \ref{org8640dba}.
\end{example}
\begin{example}
\label{sec-5-2-15}
\label{orge84dd8e}
Let \(p : \cat{E} \to \cat{B}\) be a fibred type-theoretic fibration category
with a fibred universe \(w : \tilde{W} \to W\) above \(u : \tilde{U} \to U\).
Suppose \(w\) is univalent (and so is \(u\)).
Then there is a pullback in \(\cat{B}\)
\[
\begin{tikzcd}
 \tilde{U} \times \tilde{U} \arrow[d,twoheadrightarrow,"u \times u"'] \arrow[r] & \tilde{U} \arrow[d,twoheadrightarrow] \\
 U \times U \arrow[r] & U \end{tikzcd}
\]
because \(u \times u\) is the composition of \(u \times 1\) and \(1 \times u\)
which are \(u\)-small fibrations.
By Proposition \ref{org4a92ba1} and \ref{orgc52c97e},
the relational model \(\Rel(p)\) has a univalent universe.
\end{example}

\label{sec-5-2-16}
Corollary \ref{orgb94bb63} also holds for
a Martin-L\"{o}f type theory with a univalent universe
and a fibred type-theoretic fibration category
with a fibred univalent universe.
\begin{corollary}
\label{sec-5-2-16-1}
\label{org24b33f0}
Let \(p : \cat{E} \to \cat{B}\) be a fibred type-theoretic fibration category
with a fibred univalent universe,
and \(R : \cat{T}(M) \to \cat{E}\) a logical predicate
on an interpretation \(F : \cat{T}(M) \to \cat{B}\) in \(\cat{B}\)
of a Martin-L\"{o}f type theory \(M\) with a univalent universe.
Then for any term \(a : A \vdash t(a) : B(a)\),
there exists a term \(a : FA \mid r : RA(a) \vdash \hat{t}(a; r) : RB(a, Ft(a); r)\)
in the internal language for \(p\),
where \(\hat{t}\) is the induced morphism
\(RA \to (Ft)^{*}RB\) from \(Rt : RA \to RB\).
\end{corollary}
\section{Relational Parametricity for Homotopy Type Theory}
\label{sec:orgf77f7c5}
\label{orge32f29e}
In this last section
we show a relational parametricity result for homotopy type theory.
As a corollary we show
that every closed term of type of polymorphic endofunctions on a loop space
is homotopic to some iterated concatenation of a loop.

\begin{theorem}[{Abstraction Theorem}]
\label{sec-6-1}
\label{org5aa2f07}
In the Martin-L\"{o}f type theory with univalent universe \(\U\),
empty type \(\Zero : \U\),
unit type \(\One : \U\),
two point type \(\Two : \U\),
type of natural numbers \(\N : \U\)
and unit circle \(\Sph{1} : \U\),
define a context \(\gamma : \Gamma, \gamma' : \Gamma, \rho : R_{\Gamma}(\gamma, \gamma')\)
for each context \(\Gamma\)
and a type \(\gamma : \Gamma, \gamma' : \Gamma, \rho : R_{\Gamma}(\gamma, \gamma'),
a : A(\gamma), a' : A(\gamma') \vdash R_{A}(\rho, a, a')\)
for each type \(\Gamma \vdash A\) such that:
\begin{itemize}
\item \(R_{()} \equiv ()\) for the empty context \(()\);
\item \(R_{\Gamma, A}((\gamma, a), (\gamma', a')) \equiv \rho : R_{\Gamma}(\gamma, \gamma'), r : R_{A}(\rho, a, a')\);
\item \(c : \Sigma_{A(\gamma)}B(\gamma), c' : \Sigma_{A(\gamma')}B(\gamma')
  \vdash R_{\Sigma_{A}B}(\rho, c, c') \equiv \Sigma_{r : R_{A}(\rho, \pi_{1}(c), \pi_{1}(c'))}R_{B}(r, \pi_{2}(c), \pi_{2}(c'))\);
\item \(f : \Pi_{A(\gamma)}B(\gamma), f' : \Pi_{A(\gamma')}B(\gamma')
  \vdash R_{\Pi_{A}B}(\rho, f, f') \equiv \Pi_{a : A(\gamma), a' : A(\gamma'), r : R_{A}(\rho, a, a')}R_{B}(r, fa, f'a')\);
\item \(r : R_{A}(\rho, a, a'), s : R_{A}(\rho, b, b'), p : a = b, p' : a' = b'
  \vdash R_{=_{A}}(r, s, p, p') \equiv r =_{\langle p, p' \rangle} s\);
\item \(c : C, c' : C \vdash R_{C}(c, c') \equiv c = c'\) for \(C \equiv \Zero, \One, \Two, \N, \Sph{1}\);
\item \(R_{\U}(X, X') \equiv X \to X' \to \U\).
\end{itemize}
Then for every term \(\Gamma \vdash t : A\),
there exists an associated term
\(\gamma : \Gamma, \gamma' : \Gamma, \rho : R_{\Gamma}(\gamma, \gamma')
\vdash \hat{t} : R_{A}(\rho, t(\gamma), t(\gamma'))\).
\end{theorem}

\begin{proof}
\label{sec-6-2}
Let \(\cat{T}\) be the syntactic category of the type theory.
Consider the relational model \(p : \Rel(\cat{T}) \to \cat{T}\)
for the codomain functor \(\fibcat{\arcat{\cat{T}}} \to \cat{T}\).
Note that it is also the gluing construction for
the functor \(\cat{T} \ni A \mapsto A \times A \in \cat{T}\).
By Example \ref{org0da0785} and \ref{orge84dd8e},
the total category \(\Rel(\cat{T})\)
has a univalent universe.
Syntactically, the universe in the relational model is the type family
\(A : \U, B : \U \vdash A \to B \to \U \ \type\).
It is easy to show that
\(R_{\Zero}\), \(R_{\One}\), \(R_{\Two}\), \(R_{\N}\) and \(R_{\Sph{1}}\)
are empty type, unit type, two point type, type of natural numbers and unit circle, respectively,
in the model \(\Rel(\cat{T})\)
by checking the induction principles for these types.
Hence \(R\) defines a logical predicate \(R : \cat{T} \to \Rel(\cat{T})\)
on the trivial interpretation \(\mathbf{id} : \cat{T} \to \cat{T}\).
The conclusion follows from Corollary \ref{org24b33f0}.
\end{proof}

\label{sec-6-3}
As a corollary of Theorem \ref{org5aa2f07}
we have the homotopy unicity property on
functions parametrized over the small types.

\begin{example}
\label{sec-6-3-1}
\label{org2ffa018}
We show that any closed term \(t : \Pi_{X : \U}X \to X\)
must be homotopic to the identity function,
that is, the type \(\Pi_{X : \U}\Pi_{x : X}tx = x\) is inhabited.

First we show the \emph{naturality} of \(t\),
that is, the type \(\Pi_{X : \U, Y : \U}\Pi_{f : X \to Y}\Pi_{x : X}f(tx) = t(fx)\)
is inhabited.
By Theorem \ref{org5aa2f07} we have a term
\begin{equation*}
\hat{t} : \Pi_{X : \U, Y : \U, P : X \to Y \to \U}
\Pi_{x : X, y : Y, p : P(x, y)}P(tx, ty).
\end{equation*}
For \(X : \U\), \(Y : \U\) and \(f : X \to Y\),
letting \(P(x, y) \equiv fx = y\),
we have \(\Pi_{x : X, y : Y, p : fx = y}f(tx) = ty\).
Taking \(y \equiv fx\) and \(p \equiv \refl\), we have \(\Pi_{x : X}f(tx) = t(fx)\).

Now let \(X \equiv \One\).
Then a function \(f : \One \to Y\) corresponds to
an element \(y : Y\), and thus
\(ty = t(f*) = f(t*) = f* = y\)
where \(* : \One\) is the constructor of the type \(\One\).
This argument except the existence of \(\hat{t}\)
can be done inside the type theory.
Therefore the type \(\Pi_{X : \U}\Pi_{x : X}tx = x\) is inhabited.
\end{example}
\begin{remark}
\label{sec-6-3-2}
As in \cite[Exercise 6.9]{hottbook},
the law of excluded middle violates the homotopy unicity property of polymorphic identity:
assuming the law of excluded middle for mere propositions in \(\U\),
we can construct a closed term \(t : \Pi_{X : \U}X \to X\)
such that \(t_{\Two}0 \equiv 1\) and \(t_{\Two}1 \equiv 0\)
where \(0 : \Two\) and \(1 : \Two\) are the constructors of the type \(\Two\).
Conversely the existence of such a polymorphic endofunction
implies the law of excluded middle \cite{booij2017parametricity}.
Example \ref{org2ffa018} says that the univalence axiom does not violate the homotopy unicity property on the contrast.

Since the law of excluded middle for mere propositions in \(\U\)
can be written as a closed type,
it can be assumed in a context.
Thus the homotopy unicity property of an \emph{open} term does not hold in general,
while in Atkey et al's reflexive graph model of the dependent type theory
with a universe, dependent product types and a type of natural numbers,
any term of type \(\Pi_{X : \U}X \to X\) is natural
as a consequence of the \emph{identity extension property} \cite[Theorem 2]{atkey2014relationally}.
\end{remark}
\begin{example}
\label{sec-6-3-3}
Let \(t : \Pi_{X : \U}\Pi_{x : X}x = x \to x = x\) be a closed term.
We show that \(t\) is homotopic to some iterated concatenation of a loop,
that is, the type \(\Sigma_{n : \Z}\Pi_{X : \U}\Pi_{x : X}\Pi_{p : x = x}tp = p^{n}\)
is inhabited.
Note that in a type theory with a two point type \(\Two : \U\),
a coproduct of two small types \(A : \U\) and \(B : \U\)
is defined as \(A + B \equiv \Sigma_{x : \Two}[A, B](x)\)
where \([A, B] : \Two \to \U\) is defined by recursion as
\([A, B](0) \equiv A\) and \([A, B](1) \equiv B\).
In particular, the type \(\Z : \U\) of integers is defined as
\(\Z \equiv \N + \One + \N\).

First we show the naturality of \(t\),
that is, the type
\(\Pi_{X : \U, Y : \U}\Pi_{f : X \to Y}\Pi_{x : X}\Pi_{p : x = x}f(tp) =t(fp)\)
is inhabited.
By Theorem \ref{org5aa2f07} we have a term
\begin{align*}
\hat{t} : {} & \Pi _{X : \U, Y : \U, W : X \to Y \to \U}\Pi_{x : X, y : Y, w : W(x, y)} \\
& \Pi_{p : x = x, q : y = y, \beta : w =_{\langle p, q \rangle} w}w =_{\langle tp, tq \rangle} w.
\end{align*}
For \(X : \U\), \(Y : \U\) and \(f : X \to Y\),
let \(W(x, y) \equiv f x = y\).
Then, for \(w : f x = y\), \(p : x = x\) and \(q : y = y\),
\(w =_{\langle p, q \rangle} w\)
is equivalent to \(fp \cdot w = w \cdot q\).
For \(x : X\) and \(p : x = x\),
taking \(y \equiv f x\), \(w \equiv \refl\), \(q \equiv fp\) and \(\beta \equiv \refl\),
we have \(f(tp) =t(fp)\).

Let \(Y : \U\), \(y : Y\) and \(q : y = y\)
which correspond to a function \(f : \Sph{1} \to Y\)
as \(f(\base) \equiv y\) and \(f(\Sloop) = q\),
where \(\base : \Sph{1}\) is the point constructor of \(\Sph{1}\)
and \(\Sloop : \base = \base\) is the path constructor.
Now \(t(\Sloop)\) is a loop in \(\Sph{1}\) at \(\base\).
Since \(\pi_{1}(\Sph{1}) \simeq \Z\) is provable in homotopy type theory \cite[Section 8.1]{hottbook},
\(t(\Sloop) = \Sloop^{n}\) for some integer \(n\).
Hence \(tq = t(f(\Sloop)) = f(t(\Sloop))
= f(\Sloop^{n}) = f(\Sloop)^{n} = q^{n}\).
This argument can be internalized except the existence of \(\hat{t}\).
\end{example}
\section{Conclusion and Future Work}
\label{sec:org5c1ec4f}
We conclude that
fibred type-theoretic fibration categories are useful
in the study of homotopy type theory, as seen in \Sect \ref{orge32f29e}
where we show the abstraction theorem and
the homotopy unicity property on polymorphic functions in homotopy type theory.
Fibred type-theoretic fibration categories
give a fibred categorical description for Shulman's gluing construction.
Although the relational model used in this work
can be obtained by gluing construction,
we expect that there are fibred type-theoretic fibration categories
that are not gluing constructions for any functor.

There also is a theoretical interest related to higher category theory.
Kapulkin constructed a locally cartesian closed quasi-category
from a categorical model of dependent type theory \cite{kapulkin2015locally}.
We conjecture that fibred type-theoretic fibration categories
are carried to cartesian fibrations \cite[Definition 2.4.2.1]{lurie2009higher}
by his construction.
This conjecture suggests that there is an \((\infty, 1)\)-categorical description
of logical predicates in terms of cartesian fibrations.

\section*{Acknowledgments}
\label{sec:org9aa513e}
I would like to thank my supervisor Masahito Hasegawa,
Shin-ya Katsumata and Naohiko Hoshino for discussions and
helpful feedback and corrections to drafts of this paper.
The reviewers of FSCD 2016, FoSSaCS 2017 and LICS 2017
also made a lot of helpful corrections and suggestions.
I would also like to thank
Kazuyuki Asada for helpful comments on relational parametricity,
and Peter LeFanu Lumsdaine
for telling me a related concept of fibrations of fibration categories.

This work was supported by JST ERATO Grant Number JPMJER1603, Japan.

\label{sec-8}
\bibliography{my-references}
\end{document}